\documentclass[11pt]{amsart}

\usepackage{subfiles}
\usepackage{comment}
\usepackage{float}
\usepackage{marginnote}
\usepackage{tabu}
\usepackage{euscript}
\usepackage[dvipsnames]{xcolor}          

\usepackage{hhline}

\usepackage{adjustbox}

\usepackage{mathtools}

\usepackage{graphicx}	
\usepackage{amssymb}
\usepackage{mathrsfs}
\usepackage{amsthm}
\usepackage{amsmath}
\usepackage{stmaryrd}
\usepackage{tikz}
\usepackage{tikz-cd}
\usepackage{accents}
\usepackage{upgreek}
\usepackage{enumerate}
\usepackage{bm}
\usepackage{mathtools}
\usepackage[all]{xy}
\usepackage{caption}
\usepackage{url}
\usepackage{float}
\usepackage{todonotes} 
\usepackage{colonequals}
\usepackage{bbm}
\usepackage{longtable}
\usepackage[full]{textcomp}
\usepackage[cal=cm]{mathalfa}
\usepackage{xparse}
\usepackage{comment}
\usepackage{cite}
\usepackage{hhline}

\usepackage{array}
\newcolumntype{P}[1]{>{\centering\arraybackslash}p{#1}}

\usetikzlibrary{calc}
\usetikzlibrary{fadings}
\usetikzlibrary{decorations.pathmorphing}
\usetikzlibrary{decorations.pathreplacing}
\usepackage{tikz,tikz-cd,tikz-3dplot}
\usepackage{pgfplots}
\usetikzlibrary{arrows,shadows,positioning, calc, decorations.markings, 
hobby,quotes,angles,decorations.pathreplacing,intersections,shapes}
\usetikzlibrary{fillbetween,backgrounds}

\usepackage[margin=1.05in]{geometry}

\tikzset{
  commutative diagrams/.cd, 
  arrow style=tikz, 
  diagrams={>=stealth}
}
\tikzset{
  arrow/.pic={\path[tips,every arrow/.try,->,>=#1] (0,0) -- +(0,4pt);},
  pics/arrow/.default={triangle 90}
}
\tikzset{->-/.style={decoration={
  markings,
  mark=at position .6 with {\arrow{latex}}},postaction={decorate}}
  }
\tikzset{
  c/.style={every coordinate/.try}
}

\theoremstyle{theorem}
\newtheorem{innercustomthm}{Theorem}
\newenvironment{customthm}[1]
  {\renewcommand\theinnercustomthm{#1}\innercustomthm}
  {\endinnercustomthm}
  
\theoremstyle{definition}

\theoremstyle{definition}

\makeatletter
\def\@tocline#1#2#3#4#5#6#7{\relax
  \ifnum #1>\c@tocdepth 
  \else
    \par \addpenalty\@secpenalty\addvspace{#2}%
    \begingroup \hyphenpenalty\@M
    \@ifempty{#4}{%
      \@tempdima\csname r@tocindent\number#1\endcsname\relax
    }{%
      \@tempdima#4\relax
    }%
    \parindent\z@ \leftskip#3\relax \advance\leftskip\@tempdima\relax
    \rightskip\@pnumwidth plus4em \parfillskip-\@pnumwidth
    #5\leavevmode\hskip-\@tempdima
      \ifcase #1
       \or\or \hskip 1em \or \hskip 2em \else \hskip 3em \fi%
      #6\nobreak\relax
    \dotfill\hbox to\@pnumwidth{\@tocpagenum{#7}}\par
    \nobreak
    \endgroup
  \fi}
\makeatother

\newcounter{marginnote}
\DeclareMathAlphabet{\mathpzc}{OT1}{pzc}{m}{it}

\usepackage[backref=page]{hyperref}
\hypersetup{
  colorlinks   = true,          
  urlcolor     = PineGreen,          
  linkcolor    = PineGreen,          
  citecolor   = PineGreen             
}

\usepackage{cleveref}

\theoremstyle{theorem}
\newtheorem{theorem}{Theorem}[section]

\newtheorem{lemma}[theorem]{Lemma}
\newtheorem{proposition}[theorem]{Proposition}

\theoremstyle{definition}

\newtheorem{remark}[theorem]{Remark}

\newtheorem*{runningexample*}{Running example}

\newtheorem*{aside*}{Aside}

\newtheorem{construction}[theorem]{Construction}

\newtheorem{definition}[theorem]{Definition}
\newtheorem{example}[theorem]{Example}

\newtheorem{proposition-definition}[theorem]{Proposition-Definition}

\newcommand{\rk}{\operatorname{rk}}

\DeclareMathOperator{\Hom}{Hom}

\newcommand{\Pic}{\mathscr{P}ic\,}

\newcommand{\ord}{\operatorname{ord}}
\newcommand{\codim}{\operatorname{codim}}

\newcommand{\RR}{\mathbb{R}}

\newcommand{\Gm}{\mathbb{G}_{\operatorname{m}}}

\newcommand{\bcd}{\begin{center}\begin{tikzcd}}
\newcommand{\ecd}{\end{tikzcd}\end{center}}

\newcommand{\Aaff}{\mathbb{A}}

\newcommand{\OO}{\mathcal{O}}
\newcommand{\N}{\mathbb{N}}
\newcommand{\Z}{\mathbb{Z}}
\newcommand{\Q}{\mathbb{Q}}

\newcommand{\R}{\mathbb{R}}

\newcommand{\Speck}{\operatorname{Spec}\kfield}
\newcommand{\kfield}{\Bbbk}

\newcommand{\Bcal}{\mathcal{B}}

\newcommand{\FF}{\mathbb{F}}

\newcommand{\Supp}{\operatorname{Supp}}

\newcommand{\Spec}{\operatorname{Spec}}
\newcommand{\acts}{\curvearrowright}

\newcommand{\on}{\operatorname}

\begin{document}
 
\title{Chow theory of toric variety bundles} 
\author{Francesca Carocci, Leonid Monin, Navid Nabijou}

\begin{abstract} We describe the Chow homology and cohomology of toric variety bundles, with no restrictions on the singularities of the fibre. We present the ordinary and equivariant homologies as modules over the cohomology of the base, identify the ordinary cohomology with homology-valued Minkowski weights, and identify the equivariant cohomology with cohomology-weighted piecewise polynomial functions. We describe the product structure on Minkowski weights via a fan displacement rule, and the non-equivariant limit via equivariant multiplicities. Along the way we establish relative analogues of the K\"unneth property and Kronecker duality. Applications include the balancing condition in logarithmic enumerative geometry.
\end{abstract}

\maketitle
\setcounter{tocdepth}{1}
\tableofcontents

\section*{Introduction} \noindent Toric variety bundles admit a rich theory that has been investigated in numerous contexts \cite{SankaranUma, Hofscheier, CarocciNabijou2, DodwellThesis}. They play a central role in enumerative geometry and the study of horospherical varieties. This paper determines their Chow theory.

\subsection{Results} The definition of a toric variety bundle is recalled in \Cref{sec: setup}. They are constructed by gluing trivial toric variety bundles via transition functions valued in the dense torus. This definition is flexible enough to produce an interesting theory, while rigid enough to retain desirable toric structures. Fix a toric variety bundle
\[ p \colon Y \to X \]
with base $X$ and fibre fan $\Sigma$. Building on work in the absolute setting (see \Cref{sec: background} below) we describe the four associated Chow theories: ordinary and equivariant homology and cohomology. Importantly, we do not impose any restrictions on the singularities of the fibre fan $\Sigma$.

The global structure, key results, and underlying assumptions of the paper are summarised in the following table:

\begin{table}[H]
\bgroup
\def\arraystretch{1.25}
\centerline{
\begin{tabular}{| P{2cm} | P{3.15cm} | P{1.8cm} | P{2.5cm} | P{2.15cm} | P{1.7cm} | P{2cm} |} \hline
Object & As an... & Section & Key result & $X$ & $\Sigma$ & Example \\ \hhline{|=|=|=|=|=|=|=|}
$A_\star^T Y$ & $A^\star X \! \otimes \! A^\star_T$-module & \Cref{sec: homology} & \Cref{prop: presentation equivariant homology relative} & Arbitrary & Arbitrary & --- \\ \hline
$A_\star Y$ & $A^\star X$-module & \Cref{sec: homology} & \Cref{prop: presentation ordinary homology relative} & Arbitrary & Arbitrary & --- \\ \hline
$A^\star Y$ & $A^\star X$-module & \Cref{sec: Minkowski weights} & Theorem~\ref{thm: Minkowski weights} & Smooth & Complete & Example~\ref{ex: MW} \\ \hline
$A^\star Y$ & $A^\star X$-algebra & \Cref{sec: product rule} & Theorem~\ref{thm: product rule} & Smooth & Complete & Example~\ref{ex: MW product} \\ \hline
$A^\star Y\!\!\to\!\!A_\star Y$ & --- & \Cref{sec: PD map} & \Cref{thm: PD map} & Smooth & Complete & --- \\ \hline
$A^\star_T Y$ & $A^\star X \! \otimes \! A^\star_T$-algebra & \Cref{sec: PP} & Theorem~\ref{thm: piecewise polynomials} & Smooth & Arbitrary & ---  \\ \hline
$A^\star_T Y\!\!\to\!\!A^\star Y $ & --- &\Cref{sec: PP to MW} & \Cref{prop: PP to MW} & Smooth & Complete & Example~\ref{ex: MW to PP} \\ \hline
\end{tabular}
}
\egroup
\end{table}
\begin{center} \vspace{-17pt} \textsc{Table 1.} Summary of the paper \end{center}

Along the way we establish relative forms of the K\"unneth property (\Cref{thm: relative Kunneth}) and Kronecker duality (\Cref{thm: relative Kronecker}) for toric variety bundles.

The assumptions on $\Sigma$ are sharp: already for toric varieties, completeness is required to describe the Chow cohomology in terms of Minkowski weights. Note that we impose no conditions on the singularities of $\Sigma$. The assumptions on $X$ are sufficient for applications to enumerative geometry. For $A^\star Y$ they are almost sharp, see \Cref{rmk: Minkowski weight assumptions sharp}. For $A^\star_T Y$ they are not sharp, but can be relaxed at the cost of imposing that $\Sigma$ is projective, see \Cref{rmk: PP assumptions sharp}.

We work with $\Z$ coefficients. If we instead work with $\Q$ coefficients, the same results and proofs apply with $X$ a Deligne--Mumford stack, or the coarse space thereof.

\subsection{Background} \label{sec: background} The Chow theory of toric varieties has been studied by many authors and is surprisingly rich, especially for singular toric varieties. Since our results build directly on this existing body of work, we provide a brief synopsis:
\begin{itemize}
\item In \cite{FMSS} the authors give module presentations for the ordinary and equivariant Chow homology of arbitrary toric varieties, and also establish the K\"unneth property and Kronecker duality in this context. Several preliminary results are established in the setting of general torus actions, and we make use of these in the present paper.
\item In \cite{BrionPP,BrionVergne,PaynePP} the equivariant Chow cohomology is equated with the ring of piecewise polynomial functions on the fan. The three papers proceed in increasing levels of generality, from smooth to simplicial to arbitrary toric varieties.
\item In \cite{FultonSturmfels} the ordinary Chow cohomology of a complete toric variety is equated with the ring of Minkowski weights. The product structure is described via an elegant fan displacement rule.
\item In \cite{KatzPayne} the limiting map from equivariant to ordinary Chow cohomology is described in terms of equivariant multiplicities, obtained by resolving singularities and applying equivariant localisation. This builds on earlier work \cite{Rossmann,BrionEquivariant}.
\end{itemize}
We extend all these results to toric variety bundles, as summarised in the above table. Absolute notions are replaced by relative notions, e.g. $\Z$-valued Minkowski weights are replaced by $A_\star X$-valued Minkowski weights. 

There is some previous work on the Chow theory of toric variety bundles, focusing on the cases where $\Sigma$ is smooth \cite{SankaranUma,DasguptaKhanUma,GKM} or on cycles of codimension one \cite{DodwellThesis}. In \cite{BoteroGeneralized} Chow cohomology is also computed for rational $T$-varieties of complexity one.

\subsection{Highlights} We highlight the novel aspects of our work.

\subsubsection{Mixing collection (\Cref{sec: setup})} Fix a toric variety bundle $p \colon Y \to X$ with fibrewise dense torus $T$. By definition there is a dense principal $T$-bundle $P \subseteq Y$ over $X$, the data of which is equivalent to a collection of line bundles
\begin{equation} \label{eqn: mixing collection introduction} L \colon M \to \Pic X \end{equation}
where $M$ is the character lattice of $T$. We call \eqref{eqn: mixing collection introduction} the \textbf{mixing collection}. The mixing collection determines the global twisting of the bundle and plays a central role in its Chow theory. In particular we are interested in the map
\[ \updelta \colon M \to A^1 (X) \]
recording the first Chern classes of the line bundles $L(m)$.

\subsubsection{Minkowski weights (\Cref{sec: Minkowski weights})} For bundles, a Minkowski weight of codimension $k$ is no longer an assignment of integers to cones, but rather an assignment of base classes
\[ W(\upsigma) \in A_{\dim X + \operatorname{codim} \upsigma - k}(X) \]
for all $\upsigma \in \Sigma$. These classes may be nontrivial whenever $k - \dim X \leqslant \operatorname{codim}(\upsigma) \leqslant k$ and not just when $\codim \upsigma=k$. Given $\upgamma \in A^\star Y$ the corresponding Minkowski weight is defined by $W(\upsigma) \colonequals p_\star (\upgamma \cap [Y(\upsigma)] )$ where $Y(\upsigma) \subseteq Y$ is the fibrewise toric stratum corresponding to $\upsigma \in \Sigma$. 

The balancing condition for Minkowski weights is also more intricate, using the mixing collection to intertwine classes attached to cones of different dimensions. For every $\uptau \in \Sigma$ we let $M(\uptau) = \uptau^\perp \cap M$. The balancing condition states that for every $m \in M(\uptau)$,
\[ \sum_{\substack{\upsigma \in \Sigma, \, \upsigma \supseteq \uptau \\ \dim \upsigma = \dim \uptau+1}} \langle m, n_{\upsigma \uptau} \rangle W(\upsigma) = \updelta(m) \cap W(\uptau) \]
where $n_{\upsigma \uptau} \in \upsigma \cap N$ is any lattice point whose image generates the one-dimensional lattice $N_\upsigma/N_\uptau$.

\subsubsection{Relative K\"unneth property and Kronecker duality (Sections~\ref{sec: relative Kunneth}~and~\ref{sec: relative Kronecker})} In order to establish the Minkowski weight description of the Chow cohomology, we first prove relative forms of the K\"unneth property and Kronecker duality for toric variety bundles. The relative K\"unneth property establishes a canonical isomorphism of $A^\star X^\prime$-modules
\[ A_\star X^\prime \otimes_{A^\star X} A_\star Y \cong A_\star (X^\prime \times_X Y) \]
for any toric variety bundle $p \colon Y \to X$ and any morphism $X^\prime \to X$. This is then used as input into the proof of relative Kronecker duality, which establishes a canonical isomorphism of $A^\star X$-modules:
\begin{align*} A^\star Y & \cong \Hom_{A^\star X} (A_\star Y, A_\star X) \\
\upgamma & \mapsto (V \mapsto p_\star(\upgamma \cap V)).	
\end{align*}
Once relative Kronecker duality is proved, the Minkowski weight description of $A^\star Y$ follows immediately from the presentation of $A_\star Y$ obtained in \Cref{prop: presentation ordinary homology relative}.

\subsubsection{Product rule for Minkowski weights (\Cref{sec: product rule})} Fix a generic displacement vector $v \in N$. The product of two Minkowski weights $W_1$ and $W_2$ is then given by the formula
\[ (W_1 W_2)(\uptau) = \sum_{(\upsigma_1,\upsigma_2)} c_{\upsigma_1 \upsigma_2} (W_1(\upsigma_1) \cdot W_2(\upsigma_2)) \]
where $c_{\upsigma_1 \upsigma_2} = [N : N_{\upsigma_1}+N_{\upsigma_2}]$ and the sum is over ordered pairs of cones $\upsigma_1,\upsigma_2 \supseteq \uptau$ such that
\begin{enumerate}
	\item $\upsigma_1 \cap (\upsigma_2+v) \neq \emptyset$,
	\item $\codim \upsigma_1 + \codim \upsigma_2 = \codim \uptau$.
\end{enumerate}
Implementing this formula results in a fan displacement algorithm for the product of Minkowski weights, extending \cite{FultonSturmfels}.

Since each Minkowski weight attaches classes to cones of various dimensions, this algorithm is more complicated than in the absolute setting. We overlay the fan and its displacement, and assign a class to \emph{every cell} in the resulting polyhedral complex by multiplying the classes given by $W_1$ and $W_2$. We then collapse this polyhedral complex down to the original fan and sum the classes appropriately. This is illustrated in \Cref{ex: MW product}, see in particular \eqref{eqn: product MW example displaced fan overlay} and \eqref{eqn: example product of MW}.

\subsubsection{From piecewise polynomials to Minkowski weights (\Cref{sec: PP to MW})} We identify $A^\star_T Y$ with the algebra of cohomology-weighted piecewise polynomial functions:
\[ A^\star_T Y \cong \operatorname{PP}^\star(\Sigma) \otimes_{\Z} A^\star X.\]
Given a homogeneous piecewise polynomial
\[ f \in \operatorname{PP}^k(\Sigma) \otimes \mathbbm{1}_X \subseteq A^k_T Y \]
we wish to describe its non-equivariant limit in $A^k Y$ as a Minkowski weight $W_f$ of codimension $k$. In the absolute setting, integers $W_f(\uptau)$ are associated to cones of codimension $k$ by multiplying the polynomial pieces $f_\upsigma$ of degree $k$ with certain rational functions $e_{\upsigma \uptau}$ of degree $-k$, known as equivariant multiplicities.

For bundles we no longer expect integers but rather classes, and we do not only consider cones of codimension $k$ but rather cones of codimension $\leq\!k$. The class $W_f(\uptau)$ is obtained by multiplying the polynomial pieces $f_\upsigma$ of degree $k$ with equivariant multiplicities $e_{\upsigma \uptau}$ of degree $\geq\!\!-k$, and then applying the mixing collection $\updelta$ to the resulting polynomial. The formula is:
\[ W_f(\uptau) = \updelta \left( \sum_\upsigma f_\upsigma e_{\upsigma \uptau} \right) \cap [X] \in A_\star X \]
where the sum is over maximal cones $\upsigma$ containing $\uptau$. The $f_\upsigma$ are polynomials and the $e_{\upsigma \uptau}$ are rational functions, but the sum within the brackets is in fact a polynomial: see \Cref{lem: residue sum is a polynomial} and \Cref{ex: MW to PP}.

\subsection{Applications} There are at least two prospective applications. The first is to logarithmic enumerative geometry. Toric variety bundles play a central role in logarithmic Gromov--Witten and Donaldson--Thomas theory where they arise as irreducible components of logarithmic expansions: see \cite[Section~4]{CarocciNabijou1} and \cite[Example~4.17]{CarocciNabijou2}. Current research investigates subschemes (not necessarily subcurves) of a toric variety bundle with fixed intersection behaviour along the boundary \cite{RangExpansions,LogDT,LogMNOP,CarocciNabijou1,KH-PT,KHQuot,TschanzExpansions}.

Tropicalising these subschemes produces piecewise linear subschemes of the fan, and a key problem is to describe the balancing condition satisfied by their slopes. For subcurves this is well-understood \cite{DodwellThesis} but a general description remains elusive. Our work provides a direct route to the balancing condition in all dimensions. It is important to permit singular fibres, since the most efficient logarithmic expansions are rarely smooth.

Toric variety bundles with smooth fibres also play an important role in classical Gromov--Witten theory, where they serve as intermediaries in degeneration-localisation schema \cite{BrownToricFibration,JiangTsengYou,CoatesVirasoro,OhGIT,KotoMirror}.

The second application is to toroidal horospherical varieties. These are the spherical varieties which form toroidal compactifications of homogeneous varieties $G/H$, where $H$ contains a maximal unipotent subgroup. It is known that every toroidal horospherical variety has the structure of a toric variety bundle over a generalised flag variety, see e.g. \cite[Section 3]{Hofscheier} and references therein. The Chow theory of generalised flag varieties is known, see e.g. \cite{Borel}, therefore our results determine the Chow theory of arbitrary toroidal horospherical varieties.

Another motivation comes from the study of general varieties with torus action, known as $T$-varieties \cite{AltmannGeometry}. Despite admitting a concrete description via polyhedral divisors, $T$-varieties are usually studied only in the case of low complexity, when the torus has small codimension. Toric variety bundles provide a class of well-understood $T$-varieties of unbounded complexity. Our work together with \cite{BoteroGeneralized} constitutes the first step towards the Chow theory of general $T$-varieties.

\subsection*{Acknowledgements} This project began at the EPFL, which we thank for ideal working conditions and mountain views. We thank Ana Botero, Johannes Hofscheier and Kimuars Kaveh for helpful conversations. We thank the anonymous referee for useful comments.

F.C. was supported by the Swiss National Science Foundation (SNFS) Grant PZ00P2\_208699, the EPFL Chair of Arithmetic Geometry (ARG), and the MIUR Excellence Department Project Mat-Mod@TOV, CUP E83C23000330006, awarded to the Department of Mathematics, University~of~Rome~Tor~Vergata, and also acknowledges the support of the PRIN Project ``Moduli spaces and birational geometry'' 2022L34E7W. L.M. was partially supported by Swiss National Science Foundation (SNSF) Grant 200021E\_224099 and by the Deutsche Forschungsgemeinschaft (DFG) -- Project Number 539974215.

\subsection*{Conventions and notation} We work over an algebraically closed field of characteristic zero, denoted $\kfield$. A variety is a connected, integral, separated scheme of finite type over $\kfield$. Given a group $G$ acting on a scheme $W$ we write $[W/G]$ for the stack quotient.

We assume familiarity with toric varieties \cite{FultonToric,CLS,Oda,KKMS} and intersection theory as developed by Fulton--MacPherson \cite{FultonBig}, including Chow homology (Chow groups) and Chow cohomology (bivariant Chow). Though not strictly required, familiarity with the Chow theory of toric varieties is important motivationally, see \Cref{sec: background} for a summary.

Notation is established in \Cref{sec: setup}. We strive to adopt uniform symbols when introducing Chow classes, namely:
\[ c \in A^\star X, \quad \upgamma \in A^\star Y, \quad V \in A_\star X, \quad V \in A_\star Y.\]
Besides this, we collect the most common notation in the following table:
\[
\bgroup
\def\arraystretch{1.2}
\begin{tabular}{| P{2.5cm} | P{13cm} |} \hline
Notation & Meaning \\ \hhline{|=|=|}
$N$ & a lattice \\ \hline
$M$ & the dual lattice to $N$ \\ \hline
$T$ & the algebraic torus associated to $N$ \\ \hline
$\Sigma$ & a fan with underlying lattice $N$ \\ \hline
$Z$ & the toric variety associated to $\Sigma$ \\ \hline
$X$ & the base variety \\ \hline
$L$ & the mixing collection $L \colon M \to \Pic X$ \\ \hline
$\updelta$ & the map $\updelta \colon M \to A^1 X$ recording the Chern classes of the mixing collection \\ \hline
$Y \xrightarrow{p} X$ & the toric variety bundle associated to $(\Sigma,L)$ \\ \hline
$Z(\upsigma)$ & the closed stratum $Z(\upsigma) \subseteq Z$ associated to $\upsigma \in \Sigma$ \\ \hline
$Y(\upsigma)$ & the fibrewise stratum $Y(\upsigma) \subseteq Y$ associated to $\upsigma \in \Sigma$\\ \hline
\end{tabular}
\egroup
\]


\section{Toric variety bundles} \label{sec: setup}

\subsection{Toric varieties} We assume familiarity with toric varieties \cite{FultonToric,CLS}. We use standard notation: $\Sigma$ denotes a fan in a lattice $N$ with dual lattice $M$. We write $Z$ for the associated toric variety and $T$ for the dense torus.

Given a cone $\upsigma \in \Sigma$ we let $N_\upsigma \subseteq N$ denote the saturated sublattice generated by $\upsigma \cap N$. We obtain short exact sequences of lattices
\[ 0 \to N_\upsigma \to N \to N(\upsigma) \to 0, \qquad 0 \to M(\upsigma) \to M \to M_\upsigma \to 0 \]
with $M_\upsigma$ dual to $N_\upsigma$ and $M(\upsigma)$ dual to $N(\upsigma)$. We have $M(\upsigma)=\upsigma^\perp \cap M$. The corresponding closed stratum is denoted
\[ Z(\upsigma) \hookrightarrow Z\]
and has dense torus $T(\upsigma) = N(\upsigma) \otimes_\Z \Gm = \Spec \kfield[M(\upsigma)]$.

\subsection{Toric variety bundles} \label{sec: toric variety bundles} Given an irreducible base variety $X$ we wish to construct a family 
\[ p \colon Y \to X\]
whose fibres are toric varieties. We recall the construction of \cite[Section~3.1]{CarocciNabijou2} (which also appears independently in \cite{SankaranUma}). For a comparison to alternative notions, see \cite[Section~3]{CarocciNabijou2}. The following input data will be fixed throughout.
\begin{definition} \textbf{Bundle data} over a base variety $X$ consists of:
\begin{enumerate}
\item \textbf{Fibre fan.} A fan $\Sigma$ in a lattice $N$.\smallskip
\item \textbf{Mixing collection.} For each $m \in M$ a line bundle $L(m)$ on $X$ together with compatible isomorphisms
\[ L(m_1) \otimes L(m_2) \cong L(m_1+m_2), \qquad L(0) \cong \OO_X. \]
In fancy language, $L$ is a $2$-group homomorphism (a lax monoidal functor):
\[ L \colon M \to \Pic X \]
where $\Pic X$ is the groupoid whose objects and arrows are, respectively, line bundles on $X$ and isomorphisms thereof.
\end{enumerate}
\end{definition}

\begin{remark} Given a mixing collection $L$, the associated \textbf{coarse mixing collection} $[L]$ is the map recording only the isomorphism classes of the $L(m)$:
\[ [L] \colon M \xrightarrow{L} \Pic X \to  \operatorname{Pic} X.\]
This determines the associated toric variety bundle up to (non-unique) isomorphism, and is a useful conceptual crutch. However at certain points it will be important to have chosen preferred representatives for these isomorphism classes.
\end{remark}

Fix bundle data $(\Sigma,L)$ over $X$. The fibre fan $\Sigma$ produces a toric variety $Z$ with dense torus $T=N \otimes_\Z \Gm = \Spec \kfield[M]$. The mixing collection $L$ produces a principal $T$-bundle $P \to X$.

\begin{definition} \label{def: toric variety bundle}The \textbf{toric variety bundle} associated to the bundle data $(\Sigma,L)$ is obtained via the mixing construction
\[Y \colonequals [(P\times Z) \slash T] \xrightarrow{p} [P/T] = X\]
where $T$ acts antidiagonally on $P \times Z$, so that $(tp,z) \sim (p,tz)$. We use the stack quotient, but since the action is free the morphism $p$ is representable.
\end{definition}
Toric variety bundles are equivalent to Zariski-locally trivial toric variety fibrations containing a dense principal torus bundle \cite[Lemma~3.4]{CarocciNabijou2}. They are constructed by gluing trivial bundles via transition functions valued in $T$ (as opposed to some larger subgroup of $\operatorname{Aut} Z$).

For the rest of the paper we fix bundle data $(\Sigma,L)$ over a base variety $X$ and study the associated toric variety bundle $p \colon Y \to X$.

\subsection{Mixing collection and universal bundle} \label{sec: mixing collection and universal bundle}
The mixing collection $L \colon M \to \Pic X$ determines the global twisting of the toric variety bundle, and plays a fundamental role in the forthcoming discussion. From the mixing collection we obtain divisor classes in the base, denoted:
\begin{align*} \updelta \colon M & \to A^1(X) \\
m & \mapsto \operatorname{c}_1(L(m)).\end{align*}

Consider the classifying stack $\Bcal T \colonequals [\Speck/T]$. The morphism of stacks $[Z/T] \to \Bcal T$ is the universal toric variety bundle with fibre $Z$: given any toric variety bundle $Y \to X$ with fibre $Z$, the mixing collection is equivalent to a morphism $L \colon X \to \Bcal T$ and there is a cartesian square \cite[Section~3.2]{CarocciNabijou2}:
\[
\begin{tikzcd}
Y \ar[d] \ar[r] \ar[rd,phantom,"\square"] & {[Z/T]} \ar[d] \\
X \ar[r,"L" swap] & \Bcal T.	
\end{tikzcd}
\]
The stack $[Z/T]$ is the Artin fan \cite[Section~5]{AbramovichEtAlSkeletons} of $\Sigma$ viewed as an abstract cone complex. The morphism $[Z/T] \to \Bcal T$ is equivalent to the data of an embedding of this cone complex as a fan in the lattice $N$.

\subsection{Fibrewise strata}
There is a fibrewise action $T \curvearrowright Y$ induced by the action $T \curvearrowright (P \times Z)$ on the second factor, and covering the trivial action $T \curvearrowright X$. Each cone $\upsigma\in\Sigma$ induces a closed embedding $i_\upsigma$ referred to as a toric variety subbundle
\bcd
Y(\upsigma) \arrow[hookrightarrow]{r}{i_{\upsigma}} \ar[dr,"p_\upsigma" {xshift=-12pt,yshift={-8pt}}] & Y\ar[d,"p"]\\
& X
\ecd
where $p_\upsigma \colon Y(\upsigma) \to X$ is also a toric variety bundle, obtained from the following bundle data:
\begin{enumerate}
 \item \textbf{Fibre fan.} The star fan $\operatorname{Star(\upsigma,\Sigma)}$ (see e.g. \cite[Section~3.1]{FultonToric}).
 \item \textbf{Mixing collection.} The restriction $L|_{M(\upsigma)} \colon M(\upsigma) \to \Pic X$.
\end{enumerate}
Each fibre of $p_\upsigma$ is a closure of an orbit of the action $T \curvearrowright Y$. In a local trivialisation $Y|_U \cong U \times Z$ it gives the closure of the orbit of $T \curvearrowright Z$ corresponding to $\upsigma$.

\subsection{Chow theories} Recall that we work over an algebraically closed field of characteristic zero. In this setting, there are four Chow theories of schemes: (non)equivariant (co)homology. Details may be found in the following references:
\[
\bgroup
\def\arraystretch{1.5}
\begin{tabular}{ P{3cm} | P{4cm} | P{4cm} }
& Homology & Cohomology \\ \hline
Ordinary & \cite[Chapters~1-6]{FultonBig} & \cite[Chapter~17]{FultonBig} \\
Equivariant & \cite[Section~2.2]{EdidinGrahamIntersection} & \cite[Section~2.6]{EdidinGrahamIntersection}
\end{tabular}
\egroup
\]
We will also use the Chow theory of Artin stacks as developed in \cite{Kresch}. However we will only use this for global quotient stacks, in which case the theory coincides with the equivariant theory of the prequotient.

Given a toric variety bundle $p \colon Y \to X$ the four Chow theories of $Y$ form algebraic structures over the Chow cohomology of $X$.

\begin{definition}\label{def: ordinary homology} The \textbf{ordinary Chow homology} $A_\star Y$ is a $\Z$-module. We endow it with an $A^\star X$-module structure by combining the pullback $p^\star \colon A^\star X \to A^\star Y$ with the cap product $A^\star Y \otimes_\Z A_\star Y \to A_\star Y$.
\end{definition}

\begin{definition} \label{def: ordinary cohomology} The \textbf{ordinary Chow cohomology} $A^\star Y$ is a $\Z$-algebra. We endow it with an $A^\star X$-algebra structure via the pullback $p^\star \colon A^\star X \to A^\star Y$.
\end{definition}

Turning to the equivariant theory, the $T$-equivariant Chow cohomology of a point is:
\[ A^\star_T = \operatorname{Sym} M. \]
Since the action $T \curvearrowright X$ is trivial, there is a natural isomorphism of rings $A^\star_T X = A^\star X \otimes_\Z A^\star_T$.

\begin{definition} \label{def: equivariant homology} The \textbf{equivariant Chow homology} $A_\star^T Y$ is an $A^\star_T$-module. We endow it with an $A^\star X \otimes_\Z A^\star_T$-module structure by combining the pullback $p^\star \colon A^\star_T X \to A^\star_T Y$ with the equivariant cap product $A^\star_T Y \otimes_\Z A_\star^T Y \to A_\star^T Y$.
\end{definition}

\begin{definition} \label{def: equivariant cohomology} The \textbf{equivariant Chow cohomology} $A^\star_T Y$ is an $A^\star_T$-algebra. We endow it with an $A^\star X \otimes_\Z A^\star_T$-algebra structure via the pullback $p^\star \colon A^\star_T X \to A^\star_T Y$.
\end{definition}

Summarising, we have:
\begin{gather*} 
A_\star Y \in A^\star X\text{-Mod}, \qquad \qquad A^\star Y  \in A^\star X\text{-Alg}, \\
A_\star^T Y  \in A^\star X \otimes_\Z A^\star_T \text{-Mod}, \qquad A^\star_T Y  \in A^\star X \otimes_\Z A^\star_T \text{-Alg}.
\end{gather*}


\section{Chow homology} \label{sec: homology}

\noindent \textbf{Assumptions in this section:} $X$ arbitrary, $\Sigma$ arbitrary.

\subsection{Invariant subschemes and eigenfunctions} \label{sec: subschemes and eigenfunctions}
Consider the fibrewise action $T \acts Y$. The equivariant and ordinary Chow homology will be generated by classes of $T$-invariant integral subschemes, with relations given by eigenfunctions thereon. We describe these.

\begin{proposition}\label{prop: invariant subvarieties} The $T$-invariant integral subschemes of $Y$ are precisely those of the form
\[ Y(\uptau)|_V \colonequals Y(\uptau) \times_X V \]
where $\uptau \in \Sigma$ and $V \subseteq X$ is an integral subscheme.
\end{proposition}

\begin{proof} The integral subscheme $Y(\uptau)|_V$ is $T$-invariant since locally on $V$ it takes the form $V \times Z(\uptau)$ where $Z(\uptau) \subseteq Z$ is the closed stratum in the associated toric variety. Suppose conversely we are given a $T$-invariant integral subscheme 
\[ W \subseteq Y.\]
If $Y=X \times Z \to X$ is a trivial toric variety bundle, then \cite[Lemma~3(a)]{FMSS} precisely states that $W = Y(\uptau)|_V$ for some cone $\uptau \in \Sigma$ and integral subscheme $V \subseteq X$. For the general case, choose a Zariski open subset $U \subseteq X$ over which $Y$ is trivial \cite[Lemma~3.4]{CarocciNabijou2}. Then $W \times_X U = Y(\uptau)|_{V \cap U}$ by the above, for some integral subscheme $V \subseteq X$. Since $W$ and $V$ are integral, we have dense open embeddings
\[ W \times_X U \hookrightarrow W, \qquad Y(\uptau)|_{V \cap U} \hookrightarrow Y(\uptau)|_V \]
and since the interiors coincide, we see that $W=Y(\uptau)|_V$ as required.
\end{proof}

We now turn to the eigenfunctions. Fix a $T$-invariant subvariety $Y(\uptau)|_V$ as in \Cref{prop: invariant subvarieties}. Consider the torus $T(\uptau)$ acting on the field of rational functions:
\begin{equation} \label{eqn: action on rational functions}  T(\uptau) \curvearrowright R(Y(\uptau)|_V).\end{equation}
We now describe the weight decomposition. Note that $Y(\uptau)|_V \to V$ is itself a toric variety bundle, with a mixing collection
\[ L \colon M(\uptau) \to M \to \Pic X \to \Pic V \]
where the final map is given by restriction. As in \Cref{sec: mixing collection and universal bundle} we obtain a map $\updelta \colon M(\uptau) \to A^1(V)$ by taking each line bundle to its first Chern class.

Choose a basis $m_1,\ldots,m_k$ for $M(\uptau)$ and let $L_i \colonequals L(m_i)$. Choose a nonzero rational section $s_i$ of $L_i$ and let $D_i \colonequals \operatorname{div} s_i$. This is a Cartier divisor on $V$ with $[D_i] = \operatorname{c}_1(L_i) \cap [V] = \updelta(m_i) \cap [V]$. The rational section $s_i$ is invertible when restricted to the complement of $\Supp D_i$, giving an isomorphism:
\[ s_i \colon \OO_{V\setminus \Supp D_i} \xrightarrow{\cong} L_i|_{V \setminus \Supp D_i}.\]
Inverting and removing the zero section, we obtain an isomorphism of principal bundles:
\begin{equation} \label{eqn: si isomorphism of Gm torsors} s_i^{-1} \colon L_i|_{V \setminus \Supp D_i}^\star \xrightarrow{\cong} \OO^\star_{V \setminus \Supp D_i}. \end{equation}
Consider the following open subset of $V$:
\begin{equation} \label{eqn: definition of U} U \colonequals V \setminus (\Supp D_1 \cup \cdots \cup \Supp D_k).\end{equation}
Assembling the isomorphisms \eqref{eqn: si isomorphism of Gm torsors} we obtain an isomorphism of principal $T(\uptau)$-bundles:
\begin{equation} \label{eqn: eigenfunction trivialisation principal bundle over open set} (s_1^{-1},\ldots,s_k^{-1}) \colon P(\uptau)|_U \xrightarrow{\cong} U \times T(\uptau) \end{equation}
which induces an isomorphism $Y(\uptau)|_U \cong U \times Z(\uptau)$. Passing from $U$ to $V$ we obtain an isomorphism of $T(\uptau)$-representations:
\begin{equation} \label{eqn: isomorphism rational function fields} R(Y(\uptau)|_V) \cong R(V \times Z(\uptau)).\end{equation}
Up to scaling by units, the eigenfunctions of $T(\uptau) \curvearrowright R(Z(\uptau))$ are the monomials $z^m \in R(Z(\uptau))$ for $m \in M(\uptau)$. As a direct consequence of \cite[Lemma 3(b)]{FMSS} the weight decomposition of $R(V \times Z(\uptau))$ takes the form:
\begin{equation} \label{eqn: weight decomposition product} R(V \times Z(\uptau)) = \bigoplus_{m \in M(\uptau)} p_1^\star R(V) \cdot p_2^\star z^m. \end{equation}
Given $z^m \in R(Z(\uptau))$ we let $t^m \in R(Y(\uptau)|_V)$ denote the rational function given by the composite:
\[ t^m \colon P(\uptau)|_U \xrightarrow{\cong} U \times T(\uptau) \xrightarrow{p_2} T(\uptau) \xrightarrow{z^m} \Gm.\]
\begin{lemma} \label{lem: shape of eigenfunction} Every eigenfunction $f \in R(Y(\uptau)|_V)$ can be written uniquely as
\[ f= p^\star g \cdot t^m \]
for some $g \in R(V)$ and $m \in M(\uptau)$, where $p \colon Y(\uptau)|_V \to V$ is the bundle projection. 
\end{lemma}
\begin{proof} Combine \eqref{eqn: isomorphism rational function fields} and \eqref{eqn: weight decomposition product}. \end{proof}

The above eigenfunctions give relations in the Chow homology of $Y(\uptau)|_V$.

\begin{proposition} \label{prop: eigenfunctions} Take $g \in R(V)$ and $m \in M(\uptau)$ and consider the eigenfunction $f=p^\star g \cdot t^m$ as in \Cref{lem: shape of eigenfunction}. The class of the corresponding divisor in $Y(\uptau)|_V$ is
\begin{equation} \label{eqn: eigenfunction divisor} [\operatorname{div}f] = p^\star [\operatorname{div}g] + \bigg( \sum_{\substack{\upsigma \in \Sigma, \, \upsigma \supseteq \uptau \\ \dim \upsigma = \dim \uptau+1}} \langle m,n_{\upsigma\uptau}\rangle [Y(\upsigma)|_V] \bigg) - p^\star \updelta(m) \cap [Y(\uptau)|_V]\end{equation}
where $n_{\upsigma\uptau} \in \upsigma \cap N$ is any lattice point whose image generates the one-dimensional lattice $N_\upsigma/N_\uptau$.
\end{proposition}
\begin{proof} It suffices to consider the rational function $t^m$ on $Y(\uptau)|_V$. This is invertible on the open set $P(\uptau)|_U$. By \eqref{eqn: definition of U} the complement of this open set is:
\[ Y(\uptau)|_V \setminus P(\uptau)|_U = \bigcup_{\substack{\upsigma \in \Sigma, \, \upsigma \supseteq \uptau \\ \dim \upsigma = \dim \uptau + 1}} Y(\upsigma)|_V \cup \bigcup_{i=1}^k (p^{-1}\Supp D_i).\]
We calculate the vanishing order of $t^m$ along each irreducible component of this complement.

We start with the $Y(\upsigma)|_V$. The mixing construction (\Cref{def: toric variety bundle}) extends the isomorphism \eqref{eqn: eigenfunction trivialisation principal bundle over open set} to an isomorphism of toric variety bundles over the open set $U$:
\[ Y(\uptau)|_U \cong U \times Z(\uptau).\]
This produces a morphism $Y(\uptau)|_U \to Z(\uptau)$ along which the rational function $t^m \in R(Y(\uptau)|_U)$ is the pullback of the rational function $z^m \in R(Z(\uptau))$. Crucially, $Y(\uptau)|_U$ contains the generic point of $Y(\upsigma)|_V$. The vanishing order of $t^m$ along $Y(\upsigma)|_V$ is equal to the vanishing order of $z^m$ along $Z(\upsigma)$ with the latter given by $\langle m, n_{\upsigma \uptau} \rangle$ \cite[Section~3.3]{FultonToric}. We conclude that the terms of $[\operatorname{div} t^m]$ involving the divisor classes $[Y(\upsigma)|_V]$ are:
\begin{equation} \label{eqn: component of rational function wrt fibrewise strata} \sum_{\substack{\upsigma \in \Sigma, \, \upsigma \supseteq \uptau \\ \dim \upsigma = \dim \uptau+1}} \langle m, n_{\upsigma \uptau} \rangle [Y(\upsigma)|_V]. \end{equation}
We now calculate the vanishing order along the irreducible components of $\cup_{i=1}^k (p^{-1} \Supp D_i)$. We first consider the rational functions
\[ t_i \colon P(\uptau)|_U \xrightarrow{(s_1^{-1},\ldots,s_k^{-1})} U \times T(\uptau) \xrightarrow{z^{m_i}} \Gm. \]
Since $\operatorname{div}s_i = D_i$ it follows that the terms of $[\operatorname{div} t_i]$ involving the irreducible components of $\cup_{i=1}^k (p^{-1} \Supp D_i)$ give precisely $-[p^{-1} D_i]=-p^\star[D_i]$. Writing $m=\Sigma_{i=1}^k c_i m_i$ in terms of the chosen basis, we have
\[ t^m = t_1^{c_1} \cdots t_k^{c_k} \]
and it follows that the terms of $[\operatorname{div} t^m]$ involving the irreducible components of $\cup_{i=1}^k (p^{-1} \Supp D_i)$ are:
\begin{align} \label{eqn: component of rational function wrt Di}
\nonumber -\sum_{i=1}^k c_i p^\star [D_i] & = -\sum_{i=1}^k c_i p^\star \left( \operatorname{c}_1 (L(m_i)) \cap [V] \right) \\
\nonumber & = -\sum_{i=1}^k c_i p^\star \updelta(m_i) \cap [Y(\uptau)|_V] \\
\nonumber & = -p^\star \left( \sum_{i=1}^k c_i \updelta(m_i)\right) \cap [Y(\uptau)|_V] \\
& = -p^\star \updelta(m) \cap [Y(\uptau)|_V].
\end{align}
Combining \eqref{eqn: component of rational function wrt fibrewise strata} and \eqref{eqn: component of rational function wrt Di} we conclude the result.
\end{proof}

\subsection{Presentations} \label{sec: presentation homology relative} We now use the above characterisations of invariant subvarieties and eigenfunctions to produce presentations for the Chow homologies.

\begin{theorem}[Presentation of $A_\star^T Y$] \label{prop: presentation equivariant homology relative} Define $B_\star^T(\Sigma,L,X)$ to be the following $\Z$-graded module over the ring $A^\star X \otimes_\Z A^\star_T$:
\begin{itemize}
\item \textbf{Generators.} For each $\uptau \in \Sigma$ a generator
\[ [Y(\uptau)] \]
in homological degree $\operatorname{dim}X + \operatorname{codim} \uptau$.\smallskip
\item \textbf{Relations.} For each $\uptau \in \Sigma$ and $m \in M(\uptau) \subseteq M$ a relation
\[ \sum_{\substack{\upsigma \in \Sigma, \, \upsigma \supseteq \uptau \\ \operatorname{dim}\upsigma = \operatorname{dim}\uptau+1}} \langle m, n_{\upsigma \uptau} \rangle [Y(\upsigma)] = (\updelta(m)+m) \cdot [Y(\uptau)] \]
where $\updelta(m) \in A^1 X$ and $m \in A^1_T=M$. Here $n_{\upsigma \uptau} \in \upsigma \cap N$ is any lattice point whose image generates the one-dimensional lattice $N_{\upsigma}/N_{\uptau}$.
\end{itemize}
Then there is a canonical isomorphism of $\Z$-graded $A^\star X \otimes_\Z A^\star_T$-modules:
\[ A_\star^T Y \cong B_\star^T(\Sigma,L,X) \otimes_{A^\star_T X} A_\star^T X.\]
In particular, if $X$ is smooth then there is a canonical isomorphism:
\[ A_\star^T Y \cong B_\star^T(\Sigma,L,X).\]
\end{theorem}

\begin{proof} The equivariant homology $A_\star^T Y$ is generated as an $A^\star_T$-module by classes of $T$-invariant integral subschemes \cite[Theorem~2.1]{BrionEquivariant}. By \Cref{prop: invariant subvarieties} these all take the form
\[ [Y(\uptau)|_V]\]
for some $\uptau \in \Sigma$ and integral subscheme $V \subseteq X$. We obtain a bijection of generators:
\begin{align} \label{eqn: bijection of generators}
\nonumber A_\star^T Y & \leftrightarrow B_\star^T(\Sigma,L,X) \otimes_{A^\star_T X} A_\star ^T X \\
[Y(\uptau)|_V] & \leftrightarrow [Y(\uptau)] \otimes [V].
\end{align}
We now show that this respects the relations. By \cite[Theorem~2.1]{BrionEquivariant} the relations in $A_\star^T Y$ are induced by eigenfunctions on invariant subschemes $Y(\uptau)|_V$. By \Cref{lem: shape of eigenfunction} these eigenfunctions take the form 
\[ p_{\uptau}^\star g \cdot t^m \]
for $g \in R(V)$ and $m \in M(\uptau)$. The associated relation is described in \Cref{prop: eigenfunctions}. We conclude that the submodule of relations whose quotient produces $A_\star^T Y$ is generated by two classes of relations. Elements of the first class correspond to $g \in R(V)$ and are given by $p_\uptau^\star \operatorname{div}g$, which under \eqref{eqn: bijection of generators} becomes:
\[ [Y(\uptau)] \otimes [\operatorname{div}g].\]
Elements of the second class correspond to $m \in M(\uptau)$ and are given by the final two terms on the right-hand side of \eqref{eqn: eigenfunction divisor}, which under \eqref{eqn: bijection of generators} become:
\[ \left( \sum_{\substack{\upsigma \in \Sigma, \, \upsigma \supseteq \uptau \\ \dim \upsigma = \dim \uptau+1}} \langle m,n_{\upsigma\uptau}\rangle [Y(\upsigma)] - (p^\star \updelta(m)+m) \cdot [Y(\uptau)] \right) \otimes [V].\]
We see that the bijection of generators \eqref{eqn: bijection of generators} identifies the submodules of relations. This establishes the desired isomorphism. Finally if $X$ is smooth then there is an isomorphism $A_\star^T X \cong A^\star_T X$ of $A^\star_T X$-modules, from which we deduce $A_\star^T Y \cong B_\star^T(\Sigma,L,X)$.
\end{proof}

\begin{theorem}[Presentation of $A_\star Y$] \label{prop: presentation ordinary homology relative} Define $B_\star(\Sigma,L,X)$ to be the following $\Z$-graded module over the ring $A^\star X$:
\begin{itemize}
\item \textbf{Generators.} For each $\uptau \in \Sigma$ a generator
\[ [Y(\uptau)]\]
in homological degree $\operatorname{dim} X + \operatorname{codim} \uptau$.
\item \textbf{Relations.} For each $\uptau \in \Sigma$ and $m \in M(\uptau) \subseteq M$ a relation
\begin{equation} \label{eqn: relation ordinary homology relative} \sum_{\substack{\upsigma \in \Sigma, \, \upsigma \supseteq \uptau \\ \operatorname{dim} \upsigma = \operatorname{dim}\uptau + 1}} \langle m, n_{\upsigma \uptau} \rangle [Y(\upsigma)] = \updelta(m) \cdot [Y(\uptau)] \end{equation}
where $\updelta(m) \in A^1 X$. Here $n_{\upsigma \uptau} \in \upsigma \cap N$ is any lattice point whose image generates the one-dimensional lattice $N_{\upsigma}/N_{\uptau}$.
\end{itemize}
Then there is a canonical isomorphism of $\Z$-graded $A^\star X$-modules:
\[ A_\star Y \cong B_\star(\Sigma,L,X) \otimes_{A^\star X} A_\star X.\]
In particular, if $X$ is smooth then there is a canonical isomorphism:
\[ A_\star Y \cong B_\star(\Sigma,L,X).\]
\end{theorem}

\begin{proof} This is obtained as the non-equivariant limit of \Cref{prop: presentation equivariant homology relative}. We have an isomorphism of $A^\star X$-modules
\[ A_\star Y \cong A_\star^T Y \otimes_{A^\star_T} \Z \]
where the map $A^\star_T \to \Z$ sends the equivariant weights to zero. Equivalently,
\[ A_\star Y \cong A_\star^T Y / A^{\geq 1}_T \]
which amounts to sending the $m \in A^1_T$ appearing in the relation in \Cref{prop: presentation equivariant homology relative} to zero.
\end{proof}


\section{Minkowski weights} \label{sec: Minkowski weights}

\noindent \textbf{Assumptions in this section:} $X$ smooth, $\Sigma$ complete.

\noindent In this section we describe $A^\star Y$ as a module over $A^\star X$ in terms of homology-valued Minkowski weights. In the next section (\Cref{sec: product rule}) we describe the product structure via a fan displacement rule. This builds on and generalises the description of the Chow cohomology of a toric variety in terms of Minkowski weights \cite{FultonSturmfels}.

\subsection{Minkowski weights}

\begin{definition} A \textbf{Minkowski weight of codimension $k$} on $Y$ is the data of a class
\[ W(\upsigma) \in A_{\dim X + \operatorname{codim} \upsigma - k}(X) \]
for every cone $\upsigma \in \Sigma$. Note that $W(\upsigma) \neq 0$ only if $k - \dim X \leqslant \operatorname{codim}(\upsigma) \leqslant k$, and when $X$ is a point we recover the classical condition $\operatorname{codim}(\upsigma)=k$. These assignments must satisfy the following balancing condition: for every $\uptau \in \Sigma$ and $m \in M(\uptau)$ we have
 \begin{equation} \label{eqn: balancing} \sum_{\substack{\upsigma \in \Sigma, \, \upsigma \supseteq \uptau \\ \dim \upsigma = \dim \uptau+1}} \langle m, n_{\upsigma \uptau} \rangle W(\upsigma) = \updelta(m) \cap W(\uptau) \end{equation}
 where $n_{\upsigma \uptau} \in \upsigma \cap N$ is any lattice point whose image generates the one-dimensional lattice $N_\upsigma/N_\uptau$.
 \end{definition}
 
The Minkowski weights are graded by codimension, and we write the set of all Minkowski weights as:
\[ \on{MW}^\star(\Sigma,L,X)  = \bigoplus_{k \geqslant 0} \on{MW}^k(\Sigma,L,X).\]
This is a $\Z$-graded $A^\star X$-module. Given $c \in A^\star X$ and a Minkowski weight $W \in \on{MW}^\star(\Sigma,L,X)$ the action is defined as follows:
\[ (c \cdot W)(\upsigma) \colonequals c \cap W(\upsigma).\]
This clearly respects the grading, restricting to a map
\[ A^l X \otimes_\Z \on{MW}^k(\Sigma,L,X) \to  \on{MW}^{k+l}(\Sigma,L,X) \]
for all $l,k \in \N$. Given this setup, we arrive at the main result of this section:
\begin{customthm}{A} \label{thm: Minkowski weights} There is a natural isomorphism of $\Z$-graded $A^\star X$-modules:
\begin{align*} A^\star Y & \cong \on{MW}^\star(\Sigma,L,X) \\
\upgamma & \mapsto (p_\star (\upgamma \cap [Y(\upsigma)]))_{\upsigma \in \Sigma}.	
\end{align*}
\end{customthm}

The proof follows the line of argument pursued in \cite{FMSS,FultonSturmfels}. We establish the relative K\"unneth property (\Cref{sec: relative Kunneth}), which we then use to establish relative Kronecker duality (\Cref{sec: relative Kronecker}). Combining relative Kronecker duality with the presentation of $A_\star Y$ given in \Cref{prop: presentation ordinary homology relative} we obtain the result.

\begin{remark} For Minkowski weights of codimension $\dim Y - 1$ (representing curves) the balancing condition \eqref{eqn: balancing} was discovered independently in \cite{DodwellThesis}.
\end{remark}

\begin{remark} \label{rmk: Minkowski weight assumptions sharp} The restriction to smooth bases $X$ is sufficient for applications to enumerative geometry. Strictly speaking, smoothness is used only in the proof of \Cref{lem: Kunneth commuting square}. Elsewhere it is sufficient to assume that $X$ is a Poincar\'e duality variety, i.e. that capping with the fundamental class induces an isomorphism:
\[ A^\star X \xrightarrow{\cong} A_\star X. \]
This latter assumption is sharp: Theorem~\ref{thm: Minkowski weights} fails otherwise, since for the trivial toric variety bundle with fibre a point it precisely states that $X$ is a Poincar\'e duality variety. Examples of singular Poincar\'e duality varieties includes rational curves with unibranch singularities.
\end{remark}

\begin{remark} Theorem~\ref{thm: Minkowski weights} applied to the universal toric variety bundle $[Z/T] \to \Bcal T$ (see \Cref{sec: mixing collection and universal bundle}) describes $A^\star_T Z$ as the $A^\star_T$-module of piecewise polynomial functions. Indeed we have
\[ A_\star(\Bcal T) = A_\star^T = \on{Sym} M \]
and so the weight associated to each maximal cone is simply a polynomial. The balancing condition is then precisely the statement that these polynomials are compatible with restriction to faces: it states that the difference of two adjacent maximal cone weights is divisible by the equation cutting out the common face, and so vanishes upon restriction. This ensures that the polynomials associated to the maximal cones glue to a global piecewise polynomial. \end{remark}

\begin{example} \label{ex: MW} Take $N = \Z^2$ and let $\Sigma$ be the complete fan
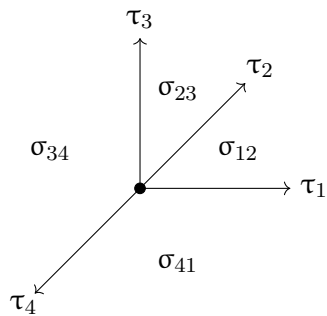
\[
\begin{tikzpicture}
	\draw(0,0) [fill=black] circle[radius=2pt];
	\draw[->] (0,0) -- (2,0);
	\draw (2,0) node[right]{$\uptau_1$};
	\draw[->] (0,0) -- (1.4,1.4);
	\draw (1.6,1.6) node{$\uptau_2$};
	\draw[->] (0,0) -- (0,2);
	\draw (0,2) node[above]{$\uptau_3$};
	\draw[->] (0,0) -- (-1.4,-1.4);
	\draw (-1.55,-1.55) node{$\uptau_4$};
	\draw (1.3,0.5) node{$\upsigma_{12}$};
	\draw (0.5,1.3) node{$\upsigma_{23}$};
	\draw (-1.2,0.5) node{$\upsigma_{34}$};
	\draw (0.5,-1) node{$\upsigma_{41}$};
\end{tikzpicture}
\]
so that $Z \cong \FF_1$. Fix a smooth base variety $X$ and an arbitrary mixing collection $L \colon \Z^2 \to \Pic X$ and let $p \colon Y \to X$ be the associated toric variety bundle. Write:
\[  \upalpha_1 \colonequals \operatorname{c}_1 (L(1,0)) \cap [X], \qquad \upalpha_2 \colonequals \operatorname{c}_1 (L(0,1)) \cap [X].\]
Writing $D_i = [Y(\uptau_i)]$ for simplicity, \Cref{prop: presentation ordinary homology relative} gives the following linear relations amongst classes of fibrewise toric divisors in $A_\star Y$:
\begin{align} \label{eqn: MW example SR relations}
\begin{split} D_1 + D_2 - D_4 = p^\star \upalpha_1, \\
D_2 + D_3 - D_4 = p^\star \upalpha_2.
\end{split}
\end{align}
Since $X$ and $Z$ are smooth, $Y$ is also smooth and hence satisfies Poincar\'e duality in Chow \cite[Corollary~17.4(a)]{FultonBig}. We now use Theorem~\ref{thm: Minkowski weights} to calculate the Minkowski weight $W_1$ Poincar\'e dual to $D_1=[Y(\uptau_1)]$. We obtain:
\begin{equation} \label{eqn: MW example 1}
\begin{tikzpicture}[baseline=(current  bounding  box.center)]
	\draw(0,0) [fill=black] circle[radius=2pt];
	\draw (0,0) node[below,blue]{$0$};
	\draw[->] (0,0) -- (2,0);
	\draw (1.9,0) node[right,blue]{$0$};
	\draw[->] (0,0) -- (1.4,1.4);
	\draw (1.65,1.65) node[blue]{$[X]$};
	\draw[->] (0,0) -- (0,2);
	\draw (0,2) node[above,blue]{$0$};
	\draw[->] (0,0) -- (-1.4,-1.4);
	\draw (-1.65,-1.65) node[blue]{$[X]$};
	\draw (1.5,0.5) node[blue]{$\upalpha_1\!-\!\upalpha_2$};
	\draw (0.5,1.3) node[blue]{$0$};
	\draw (-1.2,0.5) node[blue]{$0$};
	\draw (0.5,-1) node[blue]{$\upalpha_1\!-\!\upalpha_2$};
\end{tikzpicture}
\end{equation}
Here each cone $\upsigma \in \Sigma$ is decorated with the corresponding class $W_1(\upsigma) \in A_\star X$. These classes are calculated using the relations \eqref{eqn: MW example SR relations}. We give two sample calculations, starting with $W_1(\uptau_2)$. We have $[Y(\uptau_2)]=D_2$ and thus:
\[ W_1(\uptau_2) = p_\star(D_1 \cdot D_2) = [X].\]
We now calculate $W_1(\upsigma_{12})$. From \eqref{eqn: MW example SR relations} we obtain
\begin{align*}
& D_1^2 = D_1 \cdot (D_4-D_2+p^\star \upalpha_1) = D_1 \cdot (D_3 + p^\star \upalpha_1 - p^\star \upalpha_2) = D_1 \cdot (p^\star \upalpha_1 - p^\star \upalpha_2) \\
\Rightarrow \ & D_1^2 D_2 = D_1 D_2 \cdot (p^\star \upalpha_1 - p^\star \upalpha_2)
\end{align*}
from which we conclude
\[ W_1(\upsigma_{12}) = p_\star(D_1 \cdot D_1 D_2) = \upalpha_1-\upalpha_2. \]
Similarly, we calculate the Minkowski weight $W_2$ Poincar\'e dual to $D_2=[Y(\uptau_2)]$ as:
\begin{equation} \label{eqn: MW example 2}
\begin{tikzpicture}[baseline=(current  bounding  box.center)]
	\draw(0,0) [fill=black] circle[radius=2pt];
	\draw (0,0) node[below,blue]{$0$};
	\draw[->] (0,0) -- (2,0);
	\draw (2,0) node[right,blue]{$[X]$};
	\draw[->] (0,0) -- (1.4,1.4);
	\draw (1.6,1.6) node[blue]{$-[X]$};
	\draw[->] (0,0) -- (0,2);
	\draw (0,2) node[above,blue]{$[X]$};
	\draw[->] (0,0) -- (-1.4,-1.4);
	\draw (-1.55,-1.55) node[blue]{$0$};
	\draw (1.3,0.5) node[blue]{$\upalpha_2$};
	\draw (0.5,1.3) node[blue]{$\upalpha_1$};
	\draw (-1.2,0.5) node[blue]{$0$};
	\draw (0.5,-1) node[blue]{$0$};
\end{tikzpicture}
\end{equation}
Finally, we illustrate the balancing condition \eqref{eqn: balancing} for $W_2$. Take:
\[ \uptau = \uptau_2, \qquad m = (1,-1) \in M(\uptau) = \uptau_2^\perp \cap M.\]
The left-hand side of \eqref{eqn: balancing} is then a sum over $\upsigma_{12}$ and $\upsigma_{23}$. We have:
\begin{align*}
n_{\upsigma_{12} \uptau_2} = \begin{psmallmatrix} 1 \\ 0 \end{psmallmatrix} & \Rightarrow \langle m, n_{\upsigma_{12} \uptau_2} \rangle W_2(\upsigma_{12}) = \upalpha_2,\\
n_{\upsigma_{23} \uptau_2} = \begin{psmallmatrix} 0 \\ 1 \end{psmallmatrix} & \Rightarrow \langle m, n_{\upsigma_{23} \uptau_2} \rangle W_2(\upsigma_{23}) = -\upalpha_1.
\end{align*}
We conclude that the left-hand side of \eqref{eqn: balancing} is $\upalpha_2-\upalpha_1$. The right-hand side is given by
\[ \updelta(1,-1) \cap W_2(\uptau_2) = - \updelta(1,-1) \cap [X] = \upalpha_2 - \upalpha_1 \]
which verifies the balancing condition.
\end{example}
 
\subsection{Relative K\"unneth property} \label{sec: relative Kunneth}  
The product of a toric variety with an arbitrary variety satisfies a K\"unneth property in Chow homology \cite[Theorem~2]{FMSS}. We establish the analogue of this fact for toric variety bundles.

\begin{theorem} \label{thm: relative Kunneth} Given a morphism $f \colon X^\prime \to X$ form the fibre product:
\[
\begin{tikzcd}
Y^\prime \ar[r,"f^\prime"] \ar[d] \ar[rd,phantom,"\square"] & Y \ar[d] \\
X^\prime \ar[r,"f"] & X.	
\end{tikzcd}
\]
There is then a canonical isomorphism of $\Z$-graded $A^\star X^\prime$-modules
\[ A_\star X^\prime \otimes_{A^\star X} A_\star Y \cong A_\star Y^\prime . \]
\end{theorem}

\begin{proof} The pullback $Y^\prime \to X^\prime$ is also a toric variety bundle, with fibre fan $\Sigma$ and mixing collection
\[ L^\prime \colon M \xrightarrow{L} \Pic X \xrightarrow{f^\star} \Pic X^\prime \]
from which we obtain $\updelta^\prime (m) = f^\star \updelta(m)$ for $m \in M$. The theorem then follows from \Cref{prop: presentation ordinary homology relative}. Since $\updelta^\prime = f^\star \circ \updelta$ we have
\[ B_\star(\Sigma,L^\prime,X^\prime) \cong B_\star(\Sigma,L,X) \otimes_{A^\star X} A^\star X^\prime.\]
Moreover since $X$ is smooth we have $A_\star Y \cong B_\star(\Sigma,L,X)$. We then compute:
\begin{align*} A_\star X^\prime  \otimes_{A^\star X} A_\star Y  & = A_\star X^\prime \otimes_{A^\star X} B_\star(\Sigma,L,X) \\
& = \big( A_\star X^\prime \otimes_{A^\star X^\prime} A^\star X^\prime \big) \otimes_{A^\star X} B_\star(\Sigma,L,X) \\
& = A_\star X^\prime \otimes_{A^\star X^\prime} \big( A^\star X^\prime \otimes_{A^\star X} B_\star(\Sigma,L,X) \big) \\
& = A_\star X^\prime \otimes_{A^\star X^\prime} B_\star(\Sigma,L^\prime,X^\prime) \\
& = A_\star Y^\prime.
\end{align*}
Note that we do not require $X^\prime$ to be smooth.
\end{proof}

We now describe the relative K\"unneth isomorphism via its action on generators.
\begin{lemma} \label{lem: Kunneth on effective classes}
Consider classes $[V] \in A_\star X^\prime$ and $[Y(\upsigma)] \in A_\star Y$ where $V \subseteq X^\prime$ is an integral subscheme and $\upsigma \in \Sigma$ is a cone. Then under the canonical K\"unneth isomorphism of \Cref{thm: relative Kunneth} we have:
\[ [V] \otimes [Y(\upsigma)] \mapsto [Y^\prime(\upsigma)|_V] = [V \times_X Y(\upsigma)]. \]	
\end{lemma}

\begin{proof} This follows immediately from the presentation of $A_\star Y$ given in \Cref{prop: presentation ordinary homology relative} and used in the proof of \Cref{thm: relative Kunneth}.
\end{proof}

\begin{lemma} \label{lem: Kunneth commuting square} The K\"unneth isomorphism is $A^\star Y$-linear: given $\upgamma \in A^k Y$ the following square commutes
\begin{equation} \label{eqn: proof of Kronecker commuting square}
\begin{tikzcd}
A_m Y^\prime \ar[r,"\cong"] \ar[d,"f^{\prime \star} \upgamma \, \cap \, - " swap] & (A_\star X^\prime \otimes_{A^\star X} A_\star Y)_m \ar[d,"\operatorname{Id} \otimes (\upgamma \, \cap \, -)"] \\
A_{m-k} Y^\prime \ar[r,"\cong"] & (A_\star X^\prime \otimes_{A^\star X} A_\star Y)_{m-k}
\end{tikzcd}
\end{equation}
where the horizontal arrows are the canonical isomorphisms given in \Cref{thm: relative Kunneth}.
\end{lemma}

\begin{proof} 
Recall that $p^\prime \colon Y^\prime \to X^\prime$ is a toric variety bundle, with fibre fan $\Sigma$ and mixing collection $L^\prime = f^\star \circ L$. It suffices to prove commutativity for the generators
\[ [Y^\prime(\upsigma)|_V] \in A_m Y^\prime \]
where $\upsigma \in \Sigma$ and $V \subseteq X^\prime$ is an integral subscheme. Consider the diagram
\[
\begin{tikzcd}
Y^\prime|_V \ar[r,hook,"t"] \ar[d] \ar[rd,phantom,"\square"] & V \times Y \ar[r,"u"] \ar[d] & Y \\
X^\prime \ar[r,hook,"s"] & X^\prime \times X
\end{tikzcd}
\]
where $s= \operatorname{Id} \times f$. This is a regular embedding since $X$ is smooth \cite[B.7.3]{FultonBig}. The fibre product $Y^\prime(\upsigma)|_V = V \times_{X^\prime} Y(\upsigma)$ is of the expected dimension because $Y(\upsigma) \to X$ is flat, and it follows that:
\[ [Y^\prime(\upsigma)|_V] = s^! [V \times Y(\upsigma)] = s^! u^\star [Y(\upsigma)]. \]
Combining this with the compatibility of cap product with Gysin pullback \cite[Definition 17.1(C1)]{FultonBig} we obtain:
\[ f^{\prime\star} \upgamma \cap [Y^\prime(\upsigma)|_V] = t^\star u^\star \upgamma \cap s^! u^\star [Y(\upsigma)] = s^! u^\star (\upgamma \cap [Y(\upsigma)]).	\]
Since $X$ is smooth, \Cref{prop: presentation ordinary homology relative} gives
\[ \upgamma \cap [Y(\upsigma)] = \sum_{\uptau \in \Sigma} p^\star c_\uptau \cap [Y(\uptau)]\]
for some $c_\uptau \in A^\star X$. We then have
\[ s^! u^\star (\upgamma \cap [Y(\upsigma)]) = \sum_{\uptau \in \Sigma} t^\star u^\star p^\star c_\uptau \cap s^! [V \times Y(\uptau)] = \sum_{\uptau \in \Sigma} t^\star u^\star p^\star c_\uptau \cap [Y^\prime(\uptau)|_V]. \]
From \Cref{lem: Kunneth on effective classes}, it follows that under the isomorphism $A_\star Y^\prime \cong A_\star X^\prime \otimes_{A^\star X} A_\star Y$, the class $[Y^\prime(\uptau)|_V]$ corresponds to $[V] \otimes [Y(\uptau)]$. We conclude that under this isomorphism we have
\[ f^{\prime\star} \upgamma \cap [Y^\prime(\upsigma)|_V] = [V] \otimes \left( \sum_{\uptau \in \Sigma} p^\star c_\uptau \cap [Y(\uptau)] \right) = [V] \otimes (\upgamma \cap [Y(\upsigma)]) \]
and therefore \eqref{eqn: proof of Kronecker commuting square} commutes.
\end{proof}

\subsection{Relative Kronecker duality} \label{sec: relative Kronecker} The Chow theory of a complete toric variety satisfies Kronecker duality \cite[Theorem~3]{FMSS}. We now use the relative K\"unneth property to establish relative Kronecker duality:
\begin{theorem} \label{thm: relative Kronecker} The relative Kronecker duality map
\begin{align} \label{eqn: Kronecker duality map}
A^\star Y & \to \operatorname{Hom}_{A^\star X} (A_\star Y, A_\star X) \\
\nonumber \upgamma & \mapsto \left( h_\upgamma \colon V \mapsto p_\star(\upgamma \cap V) \right)
\end{align}
is an isomorphism of $\Z$-graded $A^\star X$-modules.
\end{theorem}

\begin{lemma} \label{lem: Kronecker duality is linear} The relative Kronecker duality map is $A^\star X$-linear and respects the $\Z$-grading. \end{lemma}

\begin{proof} Given $c \in A^\star X$, $\upgamma \in A^\star Y$ and $V \in A_\star Y$ we have:
\[ p_\star( ( p^\star c \cdot \upgamma) \cap V) = p_\star ( p^\star c \cap (\upgamma \cap V)) = c \cap p_\star( \upgamma \cap V) \]
which precisely shows that the Kronecker duality map is $A^\star X$-linear. It is clear that it respects the $\Z$-grading, where the decomposition of the right hand side into graded pieces is given by
\[ \operatorname{Hom}_{A^\star X} (A_\star Y, A_\star X) = \bigoplus_{k \geqslant 0} \operatorname{Hom}_{A^\star X} (A_\star Y, A_{\star-k} X). \qedhere \]
\end{proof}

We now show that \eqref{eqn: Kronecker duality map} is an isomorphism. As in the proof of the absolute case, the idea is to construct an explicit inverse. Fix a homogeneous element
\[ h \in \operatorname{Hom}_{A^\star X} (A_\star Y, A_{\star-k} X).\]
We will construct a corresponding class $\upgamma_{h}\in A^k(Y)$. By definition such an operational class consists of the data of compatible homomorphisms
\begin{equation} \label{eqn: homomorphism c phi f} \upgamma_{h}(f) \colon A_m (X^\prime)\to A_{m -k}(X^\prime) \end{equation}
for every $f \colon X^\prime \to Y$ and $m \in \N$. Given such a map $f$, we consider the composite $X^\prime \xrightarrow{f} Y \xrightarrow{p} X$ fitting into the diagram
\[
\begin{tikzcd}
Y^\prime \ar[r] \ar[d] & Y \ar[d,"p"] \\
X^\prime \ar[r] \ar[ru,"f" {xshift=2pt,yshift=-2pt}] & X
\end{tikzcd}
\]
with $Y^\prime \colonequals X^\prime \times_X Y$. The graph over $X$ of the morphism $f$ is a closed embedding $\Gamma_f \colon X^\prime \to Y^\prime$ because $Y$ is separated over $X$ \cite[Proposition~11.2.6]{Vakil}. Pushforward along $\Gamma_f$ induces a sequence
\begin{equation} \label{eqn: Kronecker proof first sequence} A_m X^\prime \xrightarrow{\Gamma_{f\star}} A_m Y^\prime \xrightarrow{\cong} (A_\star X^\prime \otimes_{A^\star X} A_\star Y)_m \end{equation}
where the isomorphism is given by \Cref{thm: relative Kunneth}. Applying $h$ gives:
\begin{equation} \label{eqn: Kronecker proof second sequence} (A_\star X^\prime \otimes_{A^\star X} A_\star Y)_m \xrightarrow{\operatorname{Id} \otimes h} (A_\star X^\prime \otimes_{A^\star X} A_\star X)_{m-k} \xrightarrow{\cong} A_{m-k} X^\prime \end{equation}
where the isomorphism follows from the fact that $X$ is smooth. Composing \eqref{eqn: Kronecker proof first sequence} and \eqref{eqn: Kronecker proof second sequence}, we obtain the desired homomorphism \eqref{eqn: homomorphism c phi f}. We now show that this provides a two-sided inverse to \eqref{eqn: Kronecker duality map}. 

\begin{lemma} \label{lem: Kronecker proof first inverse} The composite
\begin{equation} A^k Y \to \Hom_{A^\star X}(A_\star Y, A_{\star-k} X) \to A^k Y \end{equation}
is the identity.
\end{lemma}

\begin{proof} Given $\upgamma \in A^k Y$ and $f \colon X^\prime \to Y$ we must show that the homomorphism
\[ f^\star \upgamma \cap - \colon A_m X^\prime \to A_{m-k} X^\prime \]
coincides with the homomorphism $A_m X^\prime \to A_{m-k} X^\prime$ constructed by combining \eqref{eqn: Kronecker proof first sequence} and \eqref{eqn: Kronecker proof second sequence}, taking $h(V) = p_\star( \upgamma \cap V)$. We first describe $f^\star \upgamma \cap -$. Let $Y^\prime \colonequals X^\prime \times_X Y$ as above, and consider the commuting diagram:
\[
\begin{tikzcd}
& X^\prime \ar[d,"\Gamma_f" hook] \ar[rd,"f"] \ar[ld,"\operatorname{Id}" swap] & \\
X^\prime & Y^\prime \ar[l,"p^\prime"] \ar[r,"f^\prime" swap] & Y.
\end{tikzcd}
\]
Given $V \in A_\star X^\prime$ we then have:
\begin{align} \label{eqn: Kronecker proof first composite identity} \nonumber f^\star \upgamma \cap V & = \operatorname{Id}_\star ( f^\star \upgamma \cap V) \\
\nonumber & = p^\prime_\star \Gamma_{f \star} \left( \Gamma_f^\star f^{\prime \star} \upgamma \cap V \right) \\
& = p^\prime_\star (f^{\prime \star} \upgamma \cap \Gamma_{f \star} V).
\end{align}

We now describe the homomorphism $A_m X^\prime \to A_{m-k} X^\prime$ constructed by combining \eqref{eqn: Kronecker proof first sequence} and \eqref{eqn: Kronecker proof second sequence}. This homomorphism appears as the top row of the following diagram
\[
\begin{tikzcd}
A_m X^\prime \ar[r,"\Gamma_{f \star}"] & A_m Y^\prime \ar[r,"\cong"] \ar[d,"V \mapsto f^{\prime \star} \upgamma \cap V"] & (A_\star X^\prime \otimes_{A^\star X} A_\star Y)_m \ar[rr,"V \otimes W \mapsto V \otimes p_\star (\upgamma \cap W)\ \ \ "] & & (A_\star X^\prime \otimes_{A^\star X} A_\star X)_{m-k} \ar[r,"\cong"] & A_{m-k} X^\prime \\
& A_{m-k} Y^\prime \ar[r,"\cong"] & (A_\star X^\prime \otimes_{A^\star X} A_\star Y)_{m-k} \ar[rr,"V \otimes W \mapsto V \otimes p_\star W"] &  &(A_\star X^\prime \otimes_{A^\star X} A_\star X)_{m-k} \ar[u,"="] &
\end{tikzcd}
\]
which commutes by \Cref{lem: Kunneth commuting square}. We refactor the top row via the bottom row, and observe that the composite $A_{m-k}Y^\prime \to A_{m-k} X^\prime$ is precisely $p^\prime_\star$ as can be checked on the generators $[Y^\prime(\upsigma)]$. Comparing to \eqref{eqn: Kronecker proof first composite identity} we conclude that the two homomorphisms coincide, as required.
\end{proof}

\begin{lemma} \label{lem: Kronecker proof second inverse} The composite
\begin{equation} \label{eqn: Kronecker proof second inverse} \Hom_{A^\star X}(A_\star Y, A_{\star-k} X) \to A^k Y \to \Hom_{A^\star X}(A_\star Y, A_{\star-k} X) \end{equation}
is the identity.
\end{lemma}

\begin{proof}
Given a homomorphism $h$ and the associated cohomology class $\upgamma_h \in A^k Y$, we must show that for $V \in A_m Y$ we have
\begin{equation} \label{eqn: Kronecker proof phi same as capping with c phi} h(v) = p_\star (\upgamma_h \cap V)  \end{equation}
in $A_{m-k} X$. Consider the fibre product
\[
\begin{tikzcd}
Y^\prime \ar[r,"p_1"] \ar[d,"p_2" swap] & Y \ar[d,"p"] \\	
Y \ar[r,"p"] & X
\end{tikzcd}
\]
and note that $Y^\prime \to X$ is also a toric variety bundle, with the following bundle data:
\begin{enumerate}
	\item \textbf{Fibre fan.} $\Sigma \times \Sigma$.
	\item \textbf{Mixing collection.} $(L,L) \colon M \oplus M \to \Pic X$.
\end{enumerate}
There is a diagonal morphism $\Delta \colon Y \to Y^\prime$. By the construction of $\upgamma _h$ the map $V \mapsto p_\star (\upgamma_h \cap V)$ is given by the composite:
\begin{equation} \label{eqn: Kronecker proof diagonal composite} A_m Y \xrightarrow{\Delta_\star} A_m Y^\prime \cong (A_\star Y \otimes_{A^\star X} A_\star Y)_m \xrightarrow{\operatorname{Id} \otimes h} (A_\star Y \otimes_{A^\star X} A_\star X)_{m-k} \cong A_{m-k} Y \xrightarrow{p_\star} A_{m-k}X.\end{equation}
We claim that the following square commutes:
\begin{equation} \label{eqn: Kronecker proof commuting square for phi}
\begin{tikzcd}
A_m Y^\prime \ar[r,"\operatorname{Id} \otimes h"] \ar[d,"p_{2\star}"] & A_{m-k} Y \ar[d,"p_\star"] \\
A_m Y \ar[r,"h"] & A_{m-k} X.
\end{tikzcd}
\end{equation}
It suffices to check generators. Since $Y^\prime \to X$ is a toric variety bundle with fibre fan $\Sigma \times \Sigma$, and $X$ is smooth, \Cref{prop: presentation ordinary homology relative} shows that $A_\star Y^\prime$ is generated over $A^\star X$ by classes
\[ [Y^\prime(\upsigma_1 \times \upsigma_2)] \]
for $\upsigma_1,\upsigma_2 \in \Sigma$. On the other hand since $X$ is smooth we can write
\[ h(V) = c(V) \cap [X] \]
for a unique map $c \colon A_\star Y \to A^\star X$. We then have:
\[ p_\star ( \operatorname{Id} \otimes h) [Y^\prime(\upsigma_1 \times \upsigma_2)] = p_\star \big( p^\star c([Y(\upsigma_2)]) \cap [Y(\upsigma_1)] \big) = \begin{cases} c([Y(\upsigma_2)]) \cap [X] \qquad & \text{when $\upsigma_1 \in \Sigma_{(n)}$} \\ 0 \qquad & \text{otherwise.} \end{cases} \]
On the other hand, we have
\[ p_{2\star} [Y^\prime(\upsigma_1 \times \upsigma_2)] = \begin{cases} [Y(\upsigma_2)] \qquad & \text{when $\upsigma_1 \in \Sigma_{(n)}$} \\ 0 \qquad & \text{otherwise} \end{cases} \]
and when $\upsigma_1 \in \Sigma_{(n)}$ we have $h([Y(\upsigma_2)]) = c([Y(\upsigma_2)]) \cap [X]$. We conclude that \eqref{eqn: Kronecker proof commuting square for phi} commutes. We use this to refactor \eqref{eqn: Kronecker proof diagonal composite}. Since $p_{2\star} \Delta_\star = \operatorname{Id}$ we precisely obtain \eqref{eqn: Kronecker proof phi same as capping with c phi} as required.\end{proof}

\begin{proof}[Proof of \Cref{thm: relative Kronecker}] Combine Lemmas~\ref{lem: Kronecker duality is linear}, \ref{lem: Kronecker proof first inverse}, and \ref{lem: Kronecker proof second inverse} above.
\end{proof}

\begin{proof}[Proof of Theorem~\ref{thm: Minkowski weights}]
This follows immediately by combining relative Kronecker duality (\Cref{thm: relative Kronecker}) and the presentation of $A_\star Y$ as an $A^\star X$-module (\Cref{prop: presentation ordinary homology relative}). The balancing condition \eqref{eqn: balancing} for Minkowski weights corresponds precisely to the relation \eqref{eqn: relation ordinary homology relative} in $A_\star Y$.
\end{proof}


\section{Product rule} \label{sec: product rule}

\noindent \textbf{Assumptions in this section:} $X$ smooth, $\Sigma$ complete.

\noindent Having described Minkowski weights as a module, we now describe their product structure in terms of a fan displacement rule, extending \cite[Theorem~4.2]{FultonSturmfels}.

Since the base $X$ is smooth, there are canonical Poincar\'e duality isomorphisms which endow $A_\star X$ with a ring structure.

\begin{customthm}{B} \label{thm: product rule}
Fix a generic vector $v \in N$. Given Minkowski weights $W_1,W_2 \in \operatorname{MW}^\star(\Sigma,L,X)$ we define the product $W_1 W_2 \in \operatorname{MW}^\star(\Sigma,L,X)$ as
\[ (W_1 W_2)(\uptau) \colonequals \sum_{(\upsigma_1,\upsigma_2)} c_{\upsigma_1 \upsigma_2} (W_1(\upsigma_1) \cdot W_2(\upsigma_2)) \]
for all $\uptau \in \Sigma$, where $c_{\upsigma_1 \upsigma_2} = [N : N_{\upsigma_1}+N_{\upsigma_2}]$. The sum is over ordered pairs of cones $\upsigma_1,\upsigma_2 \supseteq \uptau$ satisfying:
\begin{enumerate}
	\item $\upsigma_1 \cap (\upsigma_2+v) \neq \emptyset$;
	\item $\codim \upsigma_1 + \codim \upsigma_2 = \codim \uptau$.
\end{enumerate}
Then the product $W_1 W_2$ of Minkowski weights corresponds to the cup product under the isomorphism given in Theorem~\ref{thm: Minkowski weights}. In particular, the definition of the product does not depend on the choice of $v$.
\end{customthm}

We work towards the proof. In Sections~\ref{sec: subbundles} and \ref{sec: subbundle classes} we discuss toric variety subbundles and derive a formula for their Chow classes in terms of fan displacements. In \Cref{sec: diagonal subbundle} we apply this in the special case of diagonal subbundles, which we then use in \Cref{sec: product rule proof} to deduce Theorem~\ref{thm: product rule}. Along the way we give a precise definition of genericity for the vector $v \in N$. Finally in \Cref{sec: PD map} we use the diagonal subbundle formula to describe the Poincar\'e duality map.

\begin{example} \label{ex: MW product} We continue \Cref{ex: MW}. We calculated the Minkowski weights $W_1$ and $W_2$ Poincar\'e dual to the classes $D_1 = [Y(\uptau_1)]$ and $D_2=[Y(\uptau_2)]$. Using the natural isomorphism $A^\star X \cong A_\star X$ we rewrite these cohomologically as follows:
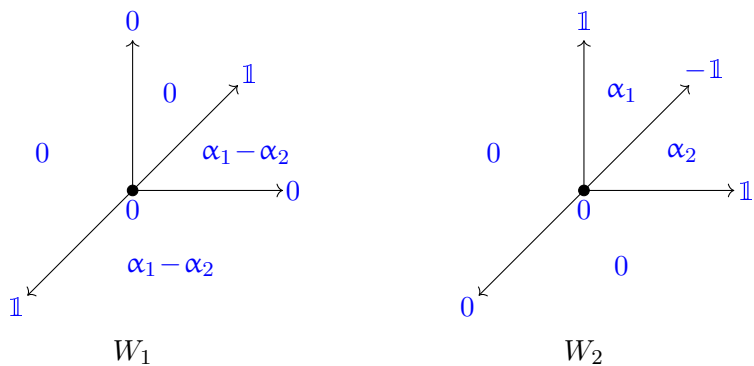
\[
\begin{tikzpicture}[baseline=(current  bounding  box.center)]
	\draw(0,0) [fill=black] circle[radius=2pt];
	\draw (0,0) node[below,blue]{$0$};
	\draw[->] (0,0) -- (2,0);
	\draw (1.9,0) node[right,blue]{$0$};
	\draw[->] (0,0) -- (1.4,1.4);
	\draw (1.55,1.55) node[blue]{$\mathbbm{1}$};
	\draw[->] (0,0) -- (0,2);
	\draw (0,2) node[above,blue]{$0$};
	\draw[->] (0,0) -- (-1.4,-1.4);
	\draw (-1.55,-1.55) node[blue]{$\mathbbm{1}$};
	\draw (1.5,0.5) node[blue]{$\upalpha_1\!-\!\upalpha_2$};
	\draw (0.5,1.3) node[blue]{$0$};
	\draw (-1.2,0.5) node[blue]{$0$};
	\draw (0.5,-1) node[blue]{$\upalpha_1\!-\!\upalpha_2$};
	\draw (0,-2.15) node{$W_1$};

	\draw(6,0) [fill=black] circle[radius=2pt];
	\draw (6,0) node[below,blue]{$0$};
	\draw[->] (6,0) -- (8,0);
	\draw (7.9,0) node[right,blue]{$\mathbbm{1}$};
	\draw[->] (6,0) -- (7.4,1.4);
	\draw (7.6,1.6) node[blue]{$-\mathbbm{1}$};
	\draw[->] (6,0) -- (6,2);
	\draw (6,2) node[above,blue]{$\mathbbm{1}$};
	\draw[->] (6,0) -- (4.6,-1.4);
	\draw (4.45,-1.55) node[blue]{$0$};
	\draw (7.3,0.5) node[blue]{$\upalpha_2$};
	\draw (6.5,1.3) node[blue]{$\upalpha_1$};
	\draw (4.8,0.5) node[blue]{$0$};
	\draw (6.5,-1) node[blue]{$0$};
	\draw (6,-2.15) node{$W_2$};
\end{tikzpicture}
\]
We now calculate the product $W_1 W_2$ using Theorem~\ref{thm: product rule}. The vector
\[ v = \begin{psmallmatrix} 2 \\ 1 \end{psmallmatrix} \in N \]
is generic, and overlaying $\Sigma$ with the displacement $\Sigma + v$ produces the following picture:
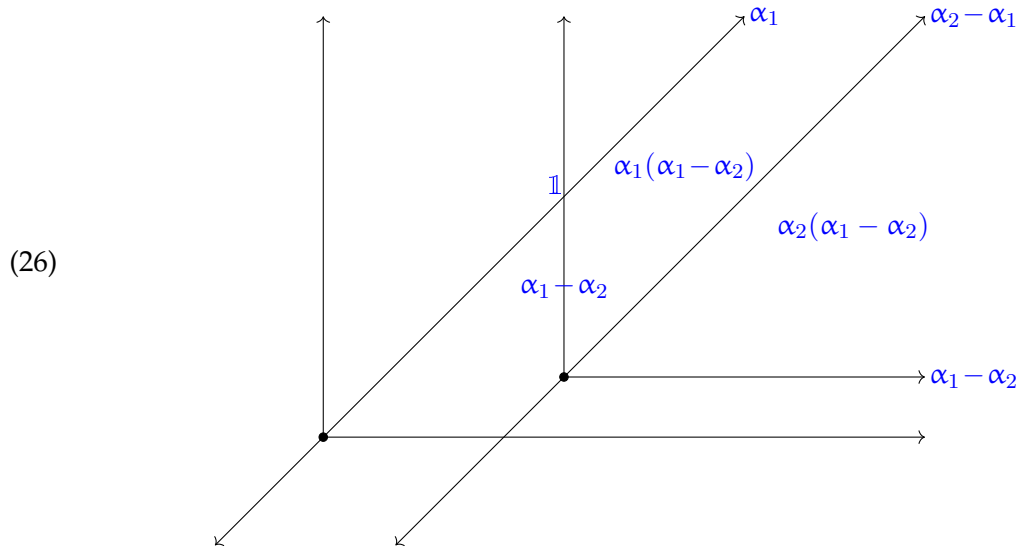
\begin{equation} \label{eqn: product MW example displaced fan overlay}
\begin{tikzpicture}[baseline=(current  bounding  box.center),scale=1.6]
	\draw(0,0) [fill=black] circle[radius=1pt];
	\draw[->] (0,0) -- (5,0);
	\draw[->] (0,0) -- (3.5,3.5);
	\draw[->] (0,0) -- (0,3.5);
	\draw[->] (0,0) -- (-0.9,-0.9);
	\draw(2,0.5) [fill=black] circle[radius=1pt];
	\draw[->] (2,0.5) -- (5,0.5);
	\draw[->] (2,0.5) -- (5,3.5);
	\draw[->] (2,0.5) -- (2,3.5);
	\draw[->] (2,0.5) -- (0.6,-0.9);
	\draw (4.95,0.5) node[right,blue]{$\upalpha_1\!-\!\upalpha_2$};
	\draw (4.4,1.75) node[blue]{$\upalpha_2(\upalpha_1-\upalpha_2)$};
	\draw (4.95,3.5) node[right,blue]{$\upalpha_2\!-\!\upalpha_1$};
	\draw (3,2.25) node[blue]{$\upalpha_1(\upalpha_1\!-\!\upalpha_2)$};
	\draw (2,1.25) node[blue]{$\upalpha_1\!-\!\upalpha_2$};
	\draw (3.45,3.5) node[right,blue]{$\upalpha_1$};
	\draw (1.925,2.1) node[blue]{$\mathbbm{1}$};
\end{tikzpicture}
\end{equation}
This is a polyhedral complex whose polyhedra correspond to ordered pairs of cones $\upsigma_1,\upsigma_2 \in \Sigma$ such that $\upsigma_1 \cap (\upsigma_2 + v) \neq \emptyset$. We have labelled each polyhedron with the term
\[ c_{\upsigma_1 \upsigma_2} (W_1(\upsigma_1) \cdot W_2(\upsigma_2)) \]
appearing in the statement of Theorem~\ref{thm: product rule}. In this example we always have $c_{\upsigma_1 \upsigma_2}=1$. Unlabelled polyhedra are those for which the associated term is zero, in this case because either $W_1(\upsigma_1)$ or $W_2(\upsigma_2)$ is zero.

To calculate the Minkowski weight $W_1 W_2$, it remains to sum over the pairs $(\upsigma_1,\upsigma_2)$ associated to each cone $\uptau$. To achieve this, we take the above picture and undo the displacement, passing to the limit $v \rightarrow 0$. This gives rise to the following picture:
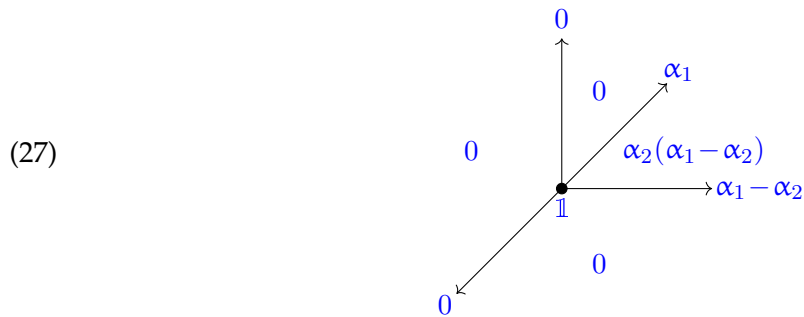
\begin{equation} \label{eqn: example product of MW}
\begin{tikzpicture}[baseline=(current  bounding  box.center)]
	\draw(0,0) [fill=black] circle[radius=2pt];
	\draw (0,0) node[below,blue]{$\mathbbm{1}$};
	\draw[->] (0,0) -- (2,0);
	\draw (1.9,0) node[right,blue]{$\upalpha_1\!-\!\upalpha_2$};
	\draw[->] (0,0) -- (1.4,1.4);
	\draw (1.55,1.55) node[blue]{$\upalpha_1$};
	\draw[->] (0,0) -- (0,2);
	\draw (0,2) node[above,blue]{$0$};
	\draw[->] (0,0) -- (-1.4,-1.4);
	\draw (-1.55,-1.55) node[blue]{$0$};
	\draw (1.75,0.5) node[blue]{$\upalpha_2(\upalpha_1\!-\!\upalpha_2)$};
	\draw (0.5,1.3) node[blue]{$0$};
	\draw (-1.2,0.5) node[blue]{$0$};
	\draw (0.5,-1) node[blue]{$0$};
\end{tikzpicture}
\end{equation}
The classes in the limit are determined as follows. For each cone $\uptau \in \Sigma$ we sum together the classes in \eqref{eqn: product MW example displaced fan overlay} associated to polyhedra which limit to $\uptau$. This corresponds to the condition that $\upsigma_1,\upsigma_2 \supseteq \uptau$. However, we discard classes associated to polyhedra which drop dimension in the limit, for instance the line segment in \eqref{eqn: product MW example displaced fan overlay} with associated class $\upalpha_1-\upalpha_2$. This corresponds to the condition that $\codim \upsigma_1 + \codim \upsigma_2 = \codim \uptau$.

Using the relations \eqref{eqn: MW example SR relations} one can verify that the Minkowski weight \eqref{eqn: example product of MW} is indeed Poincar\'e dual to $D_1 D_2$.
\end{example}

\subsection{Toric variety subbundles} \label{sec: subbundles} Fix a fan $\Sigma^\prime$ with lattice $N^\prime$ and consider a saturated sublattice
\[ N \subseteq N^\prime.\]
This induces a subtorus $T \subseteq T^\prime$ and a fan $\Sigma = \Sigma^\prime \cap N$. The map of fans $(N, \Sigma) \to (N^\prime, \Sigma^\prime)$ induces a toric morphism
\begin{equation} \label{eqn: inclusion of toric subvariety} Z \to Z^\prime \end{equation}
which is the normalisation of the closure of the image of $T \hookrightarrow T^\prime \hookrightarrow Z^\prime$. We refer to this as a \textbf{toric subvariety}; note that toric subvarieties are distinct from toric strata.

\begin{remark} The toric morphism $Z \to Z^\prime$ may not be a closed embedding, because the closure of $T \hookrightarrow Z^\prime$ may not be normal. A simple example is given by the cuspidal cubic $y^3=x^2$ in the affine plane. Nevertheless, we will continue to refer to $Z \to Z^\prime$ as a toric subvariety, since we are exclusively interested in its Chow class. \end{remark}

Now fix bundle data $(\Sigma^\prime,L^\prime)$ with associated toric variety bundle $Y^\prime \to X$. Given a saturated sublattice $N \subseteq N^\prime$ we wish to produce a toric variety bundle $Y \to X$ and a morphism $Y \to Y^\prime$ which fibrewise restricts to \eqref{eqn: inclusion of toric subvariety}.

There is a complication. In the absolute setting the inclusion $T \hookrightarrow Z^\prime$ arises as the $T$-orbit of the identity element of $T^\prime \subseteq Z^\prime$, but in the bundle setting there is no preferred identity element of $P^\prime \subseteq Y^\prime$. To obtain a toric variety subbundle we require additional data.

\begin{definition} Given bundle data $(\Sigma^\prime,L^\prime)$, a saturated sublattice $N \subseteq N^\prime$ is \textbf{rectifiable} if the coarse mixing collection
\[ {[L^\prime]} \colon M^\prime \xrightarrow{L^\prime} \Pic X \to \operatorname{Pic} X \]
is trivial when restricted to $N^\perp \subseteq M^\prime$. This is equivalent to the existence of a (necessarily unique) factorisation $[L]$ of the coarse mixing collection through $M$:
\[ 
\begin{tikzcd} M^\prime \ar[r] \ar[rr,bend right, "{[L^\prime]}" below] & M \ar[r, dashed, "{[L]}"] & \operatorname{Pic} X.
\end{tikzcd}
\]
A \textbf{rectification} is a system of trivialisations
\[ L^\prime(m) \cong \OO_X \]
for $m \in N^\perp$ which are compatible with tensor products and the isomorphisms encoded by $L^\prime$. This is equivalent to a choice of factorisation $L$
\begin{equation} \label{eqn: subbundle factorisation of mixing collection}
\begin{tikzcd} M^\prime \ar[r] \ar[rr,bend right, "L^\prime" below] & M \ar[r,"L", dashed ] & \Pic X
\end{tikzcd}
\end{equation}
enhancing the coarse factorisation above. A \textbf{rectified} sublattice is a rectifiable sublattice equipped with a rectification.
\end{definition}

\begin{construction} \label{construction: toric subbundle} Fix bundle data $(\Sigma^\prime,L^\prime)$ and a rectified saturated sublattice $N \subseteq N^\prime$. We will produce a natural morphism of toric variety bundles
\[ \upiota \colon Y \to Y^\prime \]
which in each fibre restricts to the morphism of toric varieties $Z \to Z^\prime$. We refer to this as a \textbf{toric variety subbundle}.

The construction proceeds as follows. Consider the dense principal bundles $P \subseteq Y$ and $P^\prime \subseteq Y^\prime$ associated to the mixing collections. The rectification gives an isomorphism
\[ P^\prime \cong (P \times T^\prime)/T.\]
From this we obtain a $T$-equivariant embedding $P \hookrightarrow P^\prime$. There is also a $T$-equivariant morphism $Z \to Z^\prime$ corresponding to the fan map $\Sigma \to \Sigma^\prime$. Together these induce a morphism
\[ Y = [(P \times Z)/T] \to [(P^\prime \times Z^\prime)/T^\prime] = Y^\prime. \]
Restricting to an open set and trivialising the principal bundles, we immediately see that this restricts to $Z \to Z^\prime$ on each fibre. In particular the morphism is proper since locally on the target it is proper.

We present an alternative formulation. The factorisation $L$ of the mixing collection amounts to a lift of $X \to \Bcal T^\prime$ to $X \to \Bcal T$. This induces a lift of $Y^\prime \to [Z^\prime/T^\prime]$ to $Y^\prime \to [Z^\prime/T]$ via the following $2$-commutative diagram:
\[
\begin{tikzcd}
Y^\prime \ar[rd,dashed] \ar[rdd,bend right] \ar[rrd, bend left=25pt] & & \\
& {[Z^\prime/T]} \ar[r] \ar[d] \ar[rd,phantom,"\square"] & {[Z^\prime/T^\prime]} \ar[d] \\
& \Bcal T \ar[r] & \Bcal T^\prime.	
\end{tikzcd}
\]
We then obtain the morphism $\upiota \colon Y \to Y^\prime$ as the fibre product:
\[
\begin{gathered}[b]
\begin{tikzcd}
Y \ar[r,"\upiota"] \ar[d] \ar[rd,phantom,"\square"] & Y^\prime \ar[d] \\
{[Z/T]} \ar[r] & {[Z^\prime/T]}.
\end{tikzcd}\\[-\dp\strutbox]
\end{gathered}
\]
\end{construction}

\begin{remark} The image of $\upiota \colon Y \to Y^\prime$ depends on the choice of rectification. Given two rectifications, the corresponding subbundles differ by the action of an element of $T^\prime \acts Y^\prime$, and the set of all such subbundles forms a torsor for $T^\prime/T$. Note that the set of of all rectifications also forms a torsor for $T^\prime/T$.
\end{remark}

\begin{remark} \label{rmk: toric subbundle of affine bundle} If $\Sigma^\prime$ is the fan of $\Aaff^{\!k}$ then $Y^\prime \to X$ is a split vector bundle, and a toric variety subbundle $Y \to Y^\prime$ is produced by choosing fibrewise identity elements  for the line bundles associated to $N^\perp \subseteq M^\prime$. This is equivalent to choosing nonvanishing sections of these line bundles. These exist if and only if the line bundles are trivial, i.e. if and only if $N$ is rectifiable, and a valid choice corresponds to a rectification. \end{remark}

\subsection{Subbundle classes} \label{sec: subbundle classes} By \Cref{prop: presentation ordinary homology relative} we can write
\[ \upiota_\star [Y] = \sum_{\upsigma \in \Sigma^\prime} p^\star c_\upsigma \cap [Y^\prime(\upsigma)] \]
for some $c_\upsigma \in A^\star X$. Note that this expression is not unique. We now describe a fan displacement rule for producing a particular such expression, which has the added benefit that for all $\upsigma \in \Sigma^\prime$ we have:
\[ c_\upsigma \in A^0 X = \Z. \]

Given a vector $v \in N^\prime$ we define
\[ \Sigma^\prime(v) \colonequals \{ \upsigma \in \Sigma^\prime : \upsigma \cap (N+v) \text{ consists of precisely one point} \}.\]
A vector $v \in N^\prime$ is \textbf{generic} if and only if $\dim \upsigma = \operatorname{codim} N$ for every $\upsigma \in \Sigma^\prime(v)$. Such vectors exist for any saturated sublattice $N$, and their rational multiples form a dense open subset of $N \otimes \RR$.

\begin{theorem} \label{thm: subbundle class} Fix a generic vector $v \in N^\prime$. Then we have
\[ \upiota_\star[Y] = \sum_{\upsigma \in \Sigma^\prime(v)} c_\upsigma [Y^\prime(\upsigma)] \]
where $c_\upsigma \colonequals [N^\prime : N_\upsigma + N] \in \N$. In particular, the class does not depend on the choice of rectification.
\end{theorem}

\begin{proof} We follow \cite[proof of Lemma~4.4]{FultonSturmfels}. The case $N = N^\prime$ is trivial and we exclude it. Since $v$ is generic, $v \not\in N$ and so $N + \Z v \cong N \oplus \Z v$. A generic vector remains generic under rescaling by elements of $\Q$ and translating by elements of $N$. Consequently we may assume that $N \oplus \Z v$ is a saturated sublattice of $N^\prime$.

We first establish the result in the case $\codim N=1$. Here we have $N^\prime = N \oplus \Z v$ and we can choose a vector $m \in M^\prime$ such that:
\[ \langle m, N \rangle =0, \qquad \langle m, v \rangle = -1.\]
As in \Cref{sec: subschemes and eigenfunctions} we trivialise the toric variety bundle over an open set as follows. Choose a basis $m_1,\ldots,m_n$ for $M^\prime$ and let $L_i \colonequals L^\prime(m_i)$. Choose nontrivial rational sections $s_i$ of $L_i$, let $D_i \colonequals \operatorname{div} s_i$, and define $U \colonequals X \setminus \cup_{i=1}^n \Supp D_i$. On this open set the sections $s_i$ are invertible and produce an isomorphism
\begin{equation} \label{eqn: isomorphism Y prime restricted to U} Y^\prime|_U \cong U \times Z^\prime \end{equation}
which produces an isomorphism between rational functions on $Y^\prime$ and rational functions on $X \times Z^\prime$. The element $m \in M^\prime$ gives a rational function $z^m$ on $Z^\prime$ which we pull back to produce a rational function $t^m$ on $Y^\prime$. Now consider the rational function
\[ f \colonequals t^m-1 \]
on $Y^\prime$. Restricting to the open set $P^\prime|_U \subseteq Y^\prime$, the divisor associated to this rational function is $P|_U$, which is the restriction of the subbundle $Y$. It remains to calculate the order of vanishing along the divisors in the complement of $P^\prime|_U \subseteq Y^\prime$, namely:
\begin{itemize}
	\item The fibrewise toric divisors $Y^\prime(\uptau)$ for $\uptau \in \Sigma^\prime_{(1)}$.
	\item The irreducible components of $(\cup_{i=1}^n p^{-1} \Supp D_i)$.
\end{itemize}
We begin with the fibrewise toric divisors. Since we have the isomorphism \eqref{eqn: isomorphism Y prime restricted to U}, the calculation is equivalent to the corresponding calculation on $Z^\prime$ which is carried out in \cite{FultonSturmfels}. We have:
\[ \operatorname{ord}_{Y^\prime(\uptau)} f = \begin{cases} -c_\uptau \qquad \text{if $\uptau \in \Sigma^\prime(v)$,} \\ 0 \qquad \ \ \, \text{\ \ otherwise}. \end{cases}\]
We now consider the irreducible components of $(\cup_{i=1}^n p^{-1} \Supp D_i)$. Here the proof diverges from the absolute case, and the assumption that $N \subseteq N^\prime$ is rectified plays a crucial role. The rectification produces a trivialisation
\[ P^\prime / T \cong X \times (T^\prime/T).\]
Since $N^\perp \subseteq M^\prime$ is the character lattice of $T^\prime/T$, the element $m \in N^\perp$ gives a homomorphism $m \colon T^\prime/T \to \Gm$. Then $t^m$ arises as the following composite:
\begin{equation} \label{eqn: tm as a composite} t^m \colon P^\prime \to P^\prime/T \xrightarrow{\cong} X \times (T^\prime/T) \to T^\prime/T \xrightarrow{m} \Gm. \end{equation}
Fix an irreducible component $D$ of $\cup_{i=1}^n \Supp D_i$ and let $D^\prime \colonequals p^{-1}D$ be the corresponding irreducible component of $(\cup_{i=1}^n p^{-1}\Supp D_i)$. The generic point of $D^\prime$ is contained in $P^\prime$ and we conclude from \eqref{eqn: tm as a composite} that $\ord_{D^\prime} t^m = 0$.

In particular $t^m$ gives a well-defined rational function on $D^\prime$. This function is nonconstant: we have $\langle m ,v \rangle \neq 0$, which implies that $m \colon T^\prime/T \to \Gm$ is nonconstant, which implies that $t^m$ is nonconstant in every fibre of $p$, and $D^\prime$ is a nonempty union of such fibres. In particular we see that $t^m \not\equiv 1$ along $D^\prime$ and therefore:
\[ \operatorname{ord}_{D^\prime} f = \operatorname{min}( \operatorname{ord}_{D^\prime} t^m, \ord_{D^\prime} 1) = \operatorname{min} (0,0) = 0. \]
There is therefore no contribution from the irreducible components of $(\cup_{i=1}^n p^{-1} \Supp D_i)$: the putative correction terms vanish. This completes the proof when $\codim N = 1$. 

For the general case, fix a rectified sublattice $N \subseteq N^\prime$ of arbitrary codimension and a generic vector $v \in N^\prime$. Consider the saturated sublattice:
\[ N^+ \colonequals N \oplus \Z v. \]
There are inclusions $N \subseteq N^+ \subseteq N^\prime$. The rectification of $N \subseteq N^\prime$ induces rectifications of $N^+ \subseteq N^\prime$ and $N \subseteq N^+$ and we obtain nested toric variety subbundles:
\[ Y \xrightarrow{\upiota} Y^+ \xrightarrow{\upkappa} Y^\prime.\]
We let $\Sigma^+ = \Sigma^\prime \cap N^+$ denote the fibre fan for $Y^+$. The vector $v \in N^+$ is generic with respect to the sublattice $N \subseteq N^+$ and since this sublattice has codimension one we can apply the case already established to obtain the following relation in $A_\star Y^+$:
\begin{equation} \label{eqn: subbundle formula first pushforward} \upiota_\star [Y] = \sum_{\upsigma^+ \in \Sigma^+(v)} c_{\upsigma^+} [Y^+(\upsigma^{+})] \end{equation}
We wish to pushforward the right-hand side along $\upkappa \colon Y^+ \to Y^\prime$. By definition each $\upsigma^{+} \in \Sigma^{+}$ can be written as $\upsigma^\prime \cap N^+_{\R}$ for some $\upsigma^\prime \in \Sigma^\prime$. Choosing the minimal such $\upsigma^\prime$ we have:
\begin{equation} \label{eqn: subbundle formula cone as intersection} \upsigma^{+} = \overline{(\upsigma^\prime)^\circ \cap N^{+}_{\R}} \end{equation}
which induces an identity of sublattices
\begin{equation} \label{eqn: subbundle formula lattice intersection} N_{\upsigma^{+}} = N_{\upsigma^\prime} \cap N^{+} \subseteq N^\prime. \end{equation}
The morphism $Y^{+} \to Y^\prime$ restricts to a morphism
\[ Y^{+}(\upsigma^{+}) \to Y^\prime(\upsigma^\prime) \]
which restricts further to a morphism of principal bundles
\begin{equation} \label{eqn: map of principal bundles} P^{+}(\upsigma^{+}) \to P^\prime(\upsigma^\prime). \end{equation}
We now describe this morphism. From \eqref{eqn: subbundle formula lattice intersection} we obtain the following diagram with exact rows:
\begin{equation} \label{eqn: subbundle formula map of exact sequences}
\begin{tikzcd}
0 \ar[r] & N_{\upsigma^{+}} \ar[r] \ar[d,hook] & N^{+} \ar[r] \ar[d,hook] & N^{+}(\upsigma^{+}) \ar[r] \ar[d] & 0 \\	
0 \ar[r] & N_{\upsigma^\prime} \ar[r] & N^\prime \ar[r] & N^\prime(\upsigma^\prime) \ar[r] & 0.
\end{tikzcd}
\end{equation}
Focusing on the final column and passing to the associated algebraic tori, we have identities of principal bundles:
\[ P^{+}(\upsigma^{+}) = P^{+} / T_{\upsigma^{+}}, \qquad P^\prime(\upsigma^\prime) = P^\prime/T_{\upsigma^\prime}. \]
By the subbundle construction (\Cref{construction: toric subbundle}) there is a natural map $P^{+} \to P^\prime$ from which we obtain the composition
\[ P^{+}/T_{\upsigma^{+}} \to P^\prime/T_{\upsigma^{+}} \to P^\prime/T_{\upsigma^\prime} \]
where the final map is induced by the group homomorphism $T_{\upsigma^{+}} \to T_{\upsigma^\prime}$ arising from the first column of \eqref{eqn: subbundle formula map of exact sequences}. This gives precisely the morphism of principal bundles \eqref{eqn: map of principal bundles}.

We now show that this is a finite morphism. For this, we pass to a local trivialisation and study the homomorphism of tori $T^{+}(\upsigma^{+}) \to T^\prime(\upsigma^\prime)$ with underlying lattice homomorphism:
\begin{equation} \label{eqn: subbundle formula lattice map} N^{+}(\upsigma^{+}) \to N^\prime(\upsigma^\prime). \end{equation}
We will show that this is a finite-index inclusion. Here we use the genericity of $v$. We have
\[ N_{\upsigma^\prime}/N_{\upsigma^{+}} = N_{\upsigma^\prime}/(N_{\upsigma^\prime} \cap N^{+}) = (N_{\upsigma^\prime}+N^{+})/N^{+} \hookrightarrow N^\prime/N^{+}. \]
Applying the snake lemma to \eqref{eqn: subbundle formula map of exact sequences}, we see that \eqref{eqn: subbundle formula lattice map} has vanishing kernel, and cokernel isomorphic to the cokernel of the inclusion
\[ N_{\upsigma^\prime} + N^{+} \hookrightarrow N^\prime. \]
To show that this is finite, we use the fact that $\upsigma^{\prime} \in \Sigma^{\prime}(v)$, which ensures that
\[ N_{\upsigma^\prime} \cap N^{+} = 0 \]
and $\rk N_{\upsigma^{\prime}} = \rk N^\prime - \rk N^{+}$. We conclude that $\rk (N_{\upsigma^\prime} + N^{+}) = \rk N^\prime$ and so the index is finite as required. We conclude that the morphism $Y^{+}(\upsigma^{+}) \to Y^\prime(\upsigma^\prime)$ is finite of degree $[N^\prime : N_{\upsigma^\prime}+N^{+}]$. Pushing forward \eqref{eqn: subbundle formula first pushforward} along $Y^{+} \to Y^\prime$ we obtain the following identity in $A_\star Y^\prime$
\[ \upkappa_\star \upiota_\star [Y] = \sum_{\upsigma^\prime \in \Sigma^\prime(v)} \left( [N^\prime : N_{\upsigma^\prime}+N^{+}]\cdot [N^{+} : N_{\upsigma^{+}}+N] \right)[Y^\prime(\upsigma^\prime)] \]
where we have used the bijection between $\Sigma^{+}(v)$ and $\Sigma^\prime(v)$ given by \eqref{eqn: subbundle formula cone as intersection}. It remains to check the coefficients. We have
\begin{align*}
[N^{+} : N_{\upsigma^{+}}+N] & = [N^{+} : (N_{\upsigma^\prime} \cap N^{+}) +N] \\
& = [N^{+} : (N_{\upsigma^\prime}+N) \cap N^{+}] \\
& = [N_{\upsigma^\prime} + N + N^{+} : N_{\upsigma^\prime} + N] \\
& = [N_{\upsigma^\prime} + N^{+} : N_{\upsigma^\prime} + N]
\end{align*}
from which we conclude
\[ 
[N^\prime : N_{\upsigma^\prime}+N^{+}]\cdot [N^{+} : N_{\upsigma^{+}}+N] = [N^\prime : N_{\upsigma^\prime}+N^{+}]\cdot [N_{\upsigma^\prime} + N^{+} : N_{\upsigma^\prime} + N] = [N^\prime : N + N_{\upsigma^\prime} ]
\]
as required.
\end{proof}

\subsection{Diagonal classes} \label{sec: diagonal subbundle} We now apply \Cref{thm: subbundle class} to diagonal subbundles. Fix bundle data $(\Sigma,L)$ with associated toric variety bundle $p \colon Y \to X$ and consider the fibre product
\[ Y^\prime \colonequals Y \times_X Y.\] 
The projection $q \colon Y^\prime \to X$ is a toric variety bundle with fibre fan $\Sigma^\prime = \Sigma \times \Sigma$ in the lattice $N^\prime = N \times N$ and mixing collection given by:
\begin{equation} \label{eqn: mixing collection for square} L^\prime \colon M \times M \xrightarrow{+} M \xrightarrow{L} \Pic X.\end{equation}
We consider the diagonal sublattice $N \subseteq N^\prime$. We have
\[ N^\perp = \{ (m,-m) \colon m \in M \} \subseteq M \times M = M^\prime \]
and for $(m,-m) \in N^\perp$ we have
\[ L^\prime(m,-m) = L(m) \otimes L(-m) \cong L(0) \cong \OO_X \]
where the isomorphisms are encoded in the mixing collection. Therefore $N \subseteq N^\prime$ is rectifiable and comes with a canonical rectification. The associated toric variety subbundle (\Cref{construction: toric subbundle}) is the diagonal embedding $\Delta \colon Y \hookrightarrow Y^\prime$ appearing in the following diagram:
\begin{equation} \label{eqn: Minkowski ring structure Y prime diagram}
\begin{tikzcd}
Y \ar[rd,"\Delta"] \ar[rrd,"\operatorname{Id}", bend left] \ar[ddr,"\operatorname{Id}" swap, bend right] & & \\ 
& Y^\prime \ar[r,"p_1"] \ar[d,"p_2" swap] \ar[rd,phantom,"\square"] & Y \ar[d,"p"] \\	
& Y \ar[r,"p"] & X.
\end{tikzcd}
\end{equation}

We say that $v \in N$ is \textbf{generic} if $(v,0) \in N^\prime=N \times N$ is generic in the sense of \Cref{sec: subbundle classes}.

\begin{proposition} \label{prop: diagonal subbundle formula} Fix $\uptau \in \Sigma$. Then for any generic vector $v \in N$ we have
\[ \Delta_\star [Y(\uptau)] = \sum_{(\upsigma_1,\upsigma_2)} c_{\upsigma_1 \upsigma_2} [Y^\prime(\upsigma_1 \times \upsigma_2)] \]
where $c_{\upsigma_1 \upsigma_2} = [N : N_{\upsigma_1} + N_{\upsigma_2}]$ and the sum is over ordered pairs of cones $\upsigma_1,\upsigma_2 \supseteq \uptau$ such that:
\begin{enumerate}
	\item $\upsigma_1 \cap (\upsigma_2+v) \neq \emptyset$;
	\item $\codim \upsigma_1 + \codim \upsigma_2 = \codim \uptau$.
\end{enumerate}
\end{proposition}

\begin{proof} This follows from \Cref{thm: subbundle class}, identically to the proof of \cite[Theorem~4.2]{FultonSturmfels} (erroneously labelled as ``Proof of Theorem~3.2'').
\end{proof}

\begin{remark} In contrast to \Cref{thm: subbundle class}, the sum in \Cref{prop: diagonal subbundle formula} only requires that $\upsigma_1 \cap (\upsigma_2 + v)$ is nonempty, not that it consists of a single point. If $\uptau=0$ then the intersection being nonempty implies that it consists of a single point, but for general $\uptau$ this does not hold. The proof in the general case proceeds by first passing to the star fan of $\uptau$; in the star fan the intersection of $\upsigma_1/\uptau$ and $\upsigma_2/\uptau$ will consist of a single point, but the original intersection will typically consist of infinitely many points. This phenomena occurs already in the absolute setting \cite{FultonSturmfels}.
\end{remark}

\subsection{Product rule} \label{sec: product rule proof} Finally, we describe the product structure on Minkowski weights. We continue the notation of the previous section. Recall that since $X$ is smooth, its Chow homology carries a ring structure. By \cite[Corollary~17.4(b)]{FultonBig} this is given by
\[ V_1 \cdot V_2 = \Delta_X^! (V_1 \times V_2) \in A_\star X \]
where $V_1,V_2 \in A_\star X$, $V_1 \times V_2 \in A_\star(X \times X)$ is the external product class, and $\Delta_X \colon X \hookrightarrow X \times X$ is a regular embedding since $X$ is smooth.

\begin{definition} Apply \Cref{prop: presentation ordinary homology relative} to choose (non-unique) expressions for the diagonal classes
\begin{equation} \label{eqn: general form for diagonal class} \upiota_\star[Y(\uptau)] = \sum_{(\upsigma_1, \upsigma_2)} q^\star c_{\upsigma_1\upsigma_2}^\uptau \cap [Y^\prime(\upsigma_1 \times \upsigma_2)] \end{equation}
for each $\uptau \in \Sigma$, where $c_{\upsigma_1 \upsigma_2}^\uptau \in A^\star X$. The \textbf{product} of Minkowski weights $W_1,W_2 \in \operatorname{MW}^\star(\Sigma,L,X)$ is then given by
\[ (W_1 W_2)(\uptau) \colonequals \sum_{(\upsigma_1,\upsigma_2)} c_{\upsigma_1 \upsigma_2}^{\uptau} \cap  \left( W_1(\upsigma_1) \cdot W_2(\upsigma_2) \right) \]
for all $\uptau \in \Sigma$.
\end{definition}

\begin{proposition} \label{thm: product MW 1} The isomorphism $A^\star Y \cong \operatorname{MW}^\star(\Sigma,L,X)$ of Theorem~\ref{thm: Minkowski weights} respects the products. In particular, the product of Minkowski weights does not depend on the choices of expressions \eqref{eqn: general form for diagonal class}.
\end{proposition}

\begin{proof} For $\upgamma_1, \upgamma_2 \in A^\star Y$ we must prove that
\[ p_\star (\upgamma_1 \upgamma_2 \cap [Y(\uptau)]) = \sum_{(\upsigma_1,\upsigma_2)} c_{\upsigma_1 \upsigma_2}^\uptau \cap \Delta_X^! \big( p_\star (\upgamma_1 \cap [Y(\upsigma_1)]) \times p_\star (\upgamma_2 \cap [Y(\upsigma_2)]) \big) \]
for all $\uptau \in \Sigma$. Referring to \eqref{eqn: Minkowski ring structure Y prime diagram} we have
\begin{align*} \upgamma_1 \upgamma_2 \cap [Y(\uptau)] & = \upgamma_1 \cap \upgamma_2 \cap [Y(\uptau)] \\
& = \upgamma_1 \cap \upgamma_2 \cap 	p_{2\star} \Delta_\star [Y(\uptau)] \\
& = \upgamma_1 \cap p_{2\star} (p_2^\star \upgamma_2 \cap \Delta_\star[Y(\uptau)]) \\
& = \upgamma_1 \cap p_{2\star} \left( p_2^\star \upgamma_2 \cap \sum_{(\upsigma_1,\upsigma_2)} q^\star c_{\upsigma_1\upsigma_2}^\uptau \cap [Y^\prime(\upsigma_1 \times \upsigma_2)] \right) \\
& = \upgamma_1 \cap p_{2\star} \left( \sum_{(\upsigma_1,\upsigma_2)}  p_2^\star p^\star c_{\upsigma_1 \upsigma_2}^\uptau \cap p_2^\star \upgamma_2 \cap [Y^\prime(\upsigma_1 \times \upsigma_2)] \right) \\
& = \upgamma_1 \cap p_{2\star} \left( \sum_{(\upsigma_1,\upsigma_2)} p_2^\star p^\star c_{\upsigma_1 \upsigma_2}^\uptau \cap ( [Y(\upsigma_1)] \otimes (\upgamma_2 \cap [Y(\upsigma_2)]) ) \right)
\end{align*}
where the final equality holds by \Cref{lem: Kunneth commuting square}. The class inside the bracket is precisely
\[ p_2^\star \upgamma_2 \cap \Delta_\star [Y(\uptau)] = \Delta_\star \left( \Delta^\star p_2^\star \upgamma_2 \cap [Y(\uptau)] \right). \]
Since this is pushed forward along $\Delta$, it is invariant under the involution $Y^\prime \to Y^\prime$ which swaps the factors. We therefore obtain
\begin{align*}
\upgamma_1 \upgamma_2 \cap [Y(\uptau)] & = \upgamma_1 \cap p_{2\star} \left( \sum_{(\upsigma_1,\upsigma_2)} p_2^\star p^\star c_{\upsigma_1 \upsigma_2}^\uptau \cap ( (\upgamma_2 \cap [Y(\upsigma_2)]) \otimes [Y(\upsigma_1)] ) \right) \\
& = \sum_{(\upsigma_1,\upsigma_2)} \upgamma_1 \cap p^\star c_{\upsigma_1 \upsigma_2}^\uptau \cap p_{2\star} \big( (\upgamma_2 \cap [Y(\upsigma_2)]) \otimes [Y(\upsigma_1)] \big) \\
& = \sum_{(\upsigma_1,\upsigma_2)} p^\star c_{\upsigma_1 \upsigma_2}^\uptau \cap p_{2\star} \big( p_2^\star \upgamma_1 \cap \big( (\upgamma_2 \cap [Y(\upsigma_2)]) \otimes [Y(\upsigma_1)] \big) \big) \\
& = \sum_{(\upsigma_1,\upsigma_2)} p^\star c_{\upsigma_1 \upsigma_2}^\uptau \cap p_{2\star} \big( (\upgamma_2 \cap [Y(\upsigma_2)]) \otimes (\upgamma_1 \cap [Y(\upsigma_1)]) \big)
\end{align*}
where the final equality holds by \Cref{lem: Kunneth commuting square}. Applying $p_\star$ we obtain:
\begin{equation} \label{eqn: product structure proof 1} p_\star (\upgamma_1 \upgamma_2 \cap [Y(\uptau)]) = \sum_{(\upsigma_1, \upsigma_2)} c_{\upsigma_1 \upsigma_2}^{\uptau} \cap q_\star \big( (\upgamma_2 \cap [Y(\upsigma_2)]) \otimes (\upgamma_1 \cap [Y(\upsigma_1)]) \big).\end{equation}
Now consider the cartesian square:
\[
\begin{tikzcd}
Y^\prime \ar[r,hook] \ar[d,"q"] \ar[rd,phantom,"\square"] & Y \times Y \ar[d,"p \times p"] \\
X \ar[r,hook,"\Delta_X"] & X \times X.	
\end{tikzcd}
\]
Given effective classes $[V_1], [V_2] \in A_\star Y$ the K\"unneth isomorphism $A_\star Y \otimes_{A^\star X} A_\star Y \cong A_\star Y^\prime$ of \Cref{thm: relative Kunneth} identifies
\[ [V_1] \otimes [V_2] = [V_1 \times_X V_2] = \Delta_X^! ( [V_1] \times [V_2] ) \]
where the first equality follows from \Cref{lem: Kunneth on effective classes} and the second equality follows because $p \times p$ is flat. We obtain:
\begin{align*} q_\star \big( (\upgamma_2 \cap [Y(\upsigma_2)]) \otimes (c\upgamma_1 \cap [Y(\upsigma_1)]) \big) & = q_\star \Delta_X^! \big( (\upgamma_2 \cap [Y(\upsigma_2)]) \times (\upgamma_1 \cap [Y(\upsigma_1)]) \big) \\
& = \Delta_X^! \big( p_\star (\upgamma_2 \cap [Y(\upsigma_2)]) \times p_\star (\upgamma_1 \cap [Y(\upsigma_1)]) \big) \\
& = \Delta_X^! \big( p_\star (\upgamma_1 \cap [Y(\upsigma_1)]) \times p_\star (\upgamma_2 \cap [Y(\upsigma_2)]) \big).
\end{align*}
Plugging into \eqref{eqn: product structure proof 1} we obtain
\[ p_\star (\upgamma_1 \upgamma_2 \cap [Y(\uptau)]) = \sum_{(\upsigma_1, \upsigma_2)} c_{\upsigma_1 \upsigma_2}^{\uptau} \cap \Delta_X^! \big( p_\star (\upgamma_1 \cap [Y(\upsigma_1)]) \times p_\star (\upgamma_2 \cap [Y(\upsigma_2)]) \big) \]
as required.
\end{proof}

\begin{proof}[Proof of Theorem~\ref{thm: product rule}] Combine \Cref{prop: diagonal subbundle formula} and \Cref{thm: product MW 1}.
\end{proof}

\subsection{Poincar\'e duality map} \label{sec: PD map} Since $X$ is smooth, $Y$ is irreducible and admits a fundamental class. There is a Poincar\'e duality map:
\begin{align*}
	A^k Y & \to A_{\dim Y-k}Y \\
	\upgamma & \mapsto \upgamma \cap [Y].
\end{align*}
This is typically not an isomorphism unless $\Sigma$ is smooth. The domain and codomain are described in Theorems~\ref{thm: Minkowski weights} and \ref{prop: presentation ordinary homology relative} respectively. We express the Poincar\'e duality map in these terms.

\begin{theorem} \label{thm: PD map} Take $\upgamma \in A^k Y$ and let $W_\upgamma \in \on{MW}^k(\Sigma,L,X)$ be the associated Minkowski weight. Fix a generic vector $v \in N$. Then
\begin{equation} \label{eqn: PD map} \upgamma \cap [Y] = \sum_{(\upsigma_1,\upsigma_2)} c_{\upsigma_1 \upsigma_2} p^\star W_\upgamma(\upsigma_2) \cap [Y(\upsigma_1)] \end{equation}
where $c_{\upsigma_1 \upsigma_2} = [ N \colon N_{\upsigma_1} + N_{\upsigma_2}]$ and the sum is over ordered pairs of cones $\upsigma_1,\upsigma_2 \in \Sigma$ such that:
\begin{enumerate}
\item $\upsigma_1 \cap (\upsigma_2+v) \neq \emptyset$;
\item $\codim \upsigma_1 + \codim \upsigma_2 = \dim N$.
\end{enumerate}
\end{theorem}

\begin{proof} \Cref{sec: relative Kronecker} constructs the inverse to the relative Kronecker duality map. By definition, for a given Chow cohomology class this reconstructs the cap product operation from the associated Minkowski weight. We apply this construction in the case where $X^\prime = Y$ and $f = \on{Id}_Y$, so that $\Gamma_f = \Delta \colon Y \to Y^\prime = Y \times_X Y$. Then $\upgamma \cap -$ is given by the composite of \eqref{eqn: Kronecker proof first sequence} and \eqref{eqn: Kronecker proof second sequence}:
\[ A_m Y \xrightarrow{\Delta_\star} A_m Y^\prime \cong (A_\star Y \otimes_{A^\star X} A_\star Y)_m \xrightarrow{\on{Id} \otimes W_\upgamma} (A_\star Y \otimes_{A^\star X} A_\star X)_{m-k} \cong A_{m-k} Y.\]
Applying this to $[Y]$ and using \Cref{prop: diagonal subbundle formula} to describe $\Delta_\star[Y]$, we obtain the result.
\end{proof}

The expression \eqref{eqn: PD map} is not unique, due to the linear relations in $A_\star Y$. The choice of expression corresponds to the choice of generic vector $v \in N$.

Simplicial toric varieties furnish a large class of examples where $A^k Z$ and $A_{n-k} Z$ are abstractly isomorphic, but the Poincar\'e duality map is not an isomorphism.


\section{Piecewise polynomials}

\noindent \textbf{Assumptions in this section:} $X$ smooth, $\Sigma$ arbitrary (\Cref{sec: PP}) or $\Sigma$ complete (\Cref{sec: PP to MW}).

\subsection{Relative piecewise polynomials} \label{sec: PP} Consider the algebra $\operatorname{PP}^\star(\Sigma)$ of piecewise polynomials on $\Sigma$. By \cite[Theorem~1]{PaynePP} there is a canonical isomorphism of $\Z$-graded $A^\star_T$-algebras
\[ A^\star_T Z = \operatorname{PP}^\star(\Sigma) \]
where the $A^\star_T$-algebra structure on the right hand side is given by the inclusion of the subalgebra of global polynomials $\operatorname{Sym} M \hookrightarrow \operatorname{PP}^\star(\Sigma)$. We now extend the above isomorphism to the setting of toric variety bundles.

\begin{customthm}{C} \label{thm: piecewise polynomials} There is a canonical isomorphism of $\Z$-graded $A^\star X \otimes_\Z A^\star_T$-algebras:
\[ A^\star_T Y \cong A^\star X \otimes_\Z \operatorname{PP}^\star(\Sigma).\]	
\end{customthm}

\begin{remark} It follows that the equivariant Chow cohomology of a toric variety bundle depends only on the fibre fan, and not on the mixing collection. This is to be expected: the mixing collection only plays a role in the linear relations, which in $A^\star_T Y$ are simply used to eliminate the equivariant parameters, irrespective of the precise form of the relations. Note however that the kernel of the homomorphism $A^\star_T Y \to A^\star Y$ does depend on the mixing collection, see \Cref{sec: PP to MW}.
\end{remark}

\begin{remark} \label{rmk: PP assumptions sharp} We can remove the smoothness assumption on $X$ in Theorem~\ref{thm: piecewise polynomials}, at the cost of imposing that the fibre fan $\Sigma$ is projective. We first prove the result for smooth $\Sigma$, following the argument in \cite[Section~3]{BrionEquivariant} which makes use of the Białynicki--Birula decomposition: the key fact is that $Y(\upsigma) \hookrightarrow Y$ is a regular embedding for every maximal cone $\upsigma$, regardless of the singularities of $X$. The general case follows \cite[Section~2]{PaynePP}, making use of equivariant envelopes for toric varieties. We omit the details; the case stated above, with $X$ smooth and $\Sigma$ arbitrary, is most relevant for applications to enumerative geometry.
\end{remark}

\begin{lemma} \label{prop: quotient bundle by dense torus} There is a canonical isomorphism of stacks:
\[ [Y/T] = X \times [Z/T]. \]
\end{lemma}

\begin{proof} By definition, $Y$ is the quotient:
\[ Y = [(P \times Z)/T] \]
via the antidiagonal action. The action $T \curvearrowright Y$ is induced by the action on either of the factors. It follows that
\[ [Y/T] = [(P \times Z)/(T \times T)] = [P/T] \times [Z/T] = X \times [Z/T] \]
as required.
\end{proof}

\begin{proposition} \label{prop: Kunneth universal toric bundle} There is a canonical isomorphism of $\Z$-graded $A^\star X \otimes_\Z A^\star_T$-algebras:
\[ A^\star (X \times [Z/T]) \cong A^\star X \otimes_\Z A^\star[Z/T].\]
\end{proposition}

\begin{proof} This follows by combining the relative K\"unneth property and relative Kronecker duality. Recall that $[Z/T] \to \Bcal T$ is the universal toric variety bundle with fibre $Z$ (\Cref{sec: mixing collection and universal bundle}). Relative Kronecker duality (\Cref{thm: relative Kronecker}) gives
\[ A^\star[Z/T] = \Hom_{A^\star_T}(A_\star[Z/T], A_\star^T).\]
Pulling back along the projection $X \times \Bcal T \to \Bcal T$ we see that $X \times [Z/T]$ is a toric variety bundle over $X \times \Bcal T$, and relative Kronecker duality again gives:
\[ A^\star(X \times [Z/T]) = \Hom_{A^\star X \otimes_\Z A^\star_T}(A_\star(X \times [Z/T]),A_\star X \otimes_\Z A_\star^T).\]
Finally by the relative K\"unneth property (\Cref{thm: relative Kunneth}) we have
\begin{align*} A_\star(X \times [Z/T]) & \cong A_\star (X \times \Bcal T) \otimes_{A^\star_T} A_\star [Z/T] \\
& = A_\star X \otimes_\Z A^\star_T \otimes_{A^\star_T} A_\star[Z/T] \\
& = A_\star X \otimes_\Z A_\star[Z/T].
\end{align*}
Since $X$ is smooth there is a canonical isomorphism of $A^\star X$-modules $A_\star X \cong A^\star X$. We obtain
\begin{align*} A^\star(X \times [Z/T]) & = \Hom_{A^\star X \otimes_\Z A^\star_T} (A^\star X \otimes_\Z A_\star[Z/T], A^\star X \otimes_\Z A_\star^T) \\
& = A^\star X \otimes_\Z \Hom_{A^\star_T} (A_\star[Z/T],A_\star^T) \\
& = A^\star X \otimes_\Z A^\star[Z/T]
\end{align*}
as required.
\end{proof}

\begin{remark}
In the previous proof we have applied the relative K\"unneth property (\Cref{thm: relative Kunneth}) and relative Kronecker duality (\Cref{thm: relative Kronecker}) to toric variety bundles whose bases are smooth Artin stacks. This generalisation is straightforward in our case, as the bases we consider are global quotients: approximations in the sense of \cite{EdidinGrahamIntersection,Kresch} are given by products of high-dimensional projective spaces.
\end{remark}

\begin{proof}[Proof of Theorem~\ref{thm: piecewise polynomials}]
By \Cref{prop: quotient bundle by dense torus} we have
\[ A^\star_T Y = A^\star[Y/T] = A^\star(X \times [Z/T]) \]
and then \Cref{prop: Kunneth universal toric bundle} gives canonical isomorphisms
\[ A^\star(X \times [Z/T]) \cong A^\star X \otimes_\Z A^\star[Z/T] = A^\star X \otimes_\Z A^\star_T Z \cong A^\star X \otimes_\Z \operatorname{PP}^\star(\Sigma) \]
where the isomorphism $A^\star_T Z \cong \operatorname{PP}^\star(\Sigma)$ is given by \cite[Theorem~1]{PaynePP}.
\end{proof}

\subsection{From piecewise polynomials to Minkowski weights} \label{sec: PP to MW} For this section, we assume that the fibre fan $\Sigma$ is complete. Theorems~\ref{thm: piecewise polynomials} and \ref{thm: Minkowski weights} establish canonical isomorphisms of $A^\star X$-algebras:
\[ A^\star_T Y = \on{PP}^\star(\Sigma) \otimes A^\star X, \qquad A^\star Y = \on{MW}^\star(\Sigma,L,X). \]
There is a canonical map $A^\star_T Y \to A^\star Y$ which takes the non-equivariant limit of an equivariant class. Composing with the isomorphisms above, we obtain a map:
\begin{equation} \label{eqn: map PP to MW} \on{PP}^\star(\Sigma) \otimes A^\star X \to \on{MW}^\star(\Sigma,L,X).\end{equation}
We describe this map, following \cite{KatzPayne}. We begin by establishing background on equivariant multiplicities. Let $\on{Sym}^{\pm} M$ denote the localisation of $\on{Sym} M$ in the set of nonzero homogeneous elements. Given a maximal cone $\upsigma \in \Sigma^{\max}$ and a face $\uptau \subseteq \upsigma$, there is an \textbf{equivariant multiplicity}
\[ e_{\upsigma \uptau} \in \on{Sym}^{\pm} M \]
which is a homogeneous element of degree $-\! \on{codim}\uptau$. If the toric stratum $Z(\uptau) \subseteq Z$ is smooth, $e_{\upsigma \uptau}$ is the inverse of the equivariant Euler class of the tangent space to $Z(\uptau)$ at the fixed point corresponding to $\upsigma$. General equivariant multiplicities are defined and studied in \cite{Rossmann, BrionEquivariant, KatzPayne} and play a central role in localisation formulae.

\begin{definition}
Fix a homogeneous piecewise polynomial $f \in \on{PP}^k(\Sigma)$ and a cone $\uptau \in \Sigma$. We define the \textbf{residue sum} of $f$ with respect to $\uptau$ as follows:
\[ R_f(\uptau) \colonequals \sum_{\substack{\upsigma \in \Sigma^{\max} \\ \upsigma \supseteq \uptau}} f_\upsigma e_{\upsigma \uptau} \in \on{Sym}^{\pm} M. \]
\end{definition}
The residue sum $R_f(\uptau)$ is homogeneous of degree $k - \on{codim} \uptau$. In \cite{KatzPayne} residue sums are studied for cones with $\on{codim} \uptau=k$, in which case $R_f(\uptau)$ is homogeneous of degree zero. We are interested in more general residue sums. We note the following basic fact.

\begin{lemma} \label{lem: residue sum is a polynomial} The residue sum $R_f(\uptau)$ belongs to the subring $\on{Sym} M \subseteq \on{Sym}^\pm M$.
\end{lemma}

\begin{proof} This follows immediately from the localisation theorem applied to the action $T \curvearrowright Z(\uptau)$. Recall that $f \in \on{PP}^\star(\Sigma) = A^\star_T Z$ and that $\on{Sym} M = A^\star_T$. By the localisation theorem, the residue sum $R_f(\uptau)$ coincides with the $T$-equivariant integral of $f$ over the toric stratum $Z(\uptau) \subseteq Z$. Consequently it must lie in $A^\star_T = \on{Sym} M$. See \cite[Theorem~2]{EdidinGrahamLocalisation} for the localisation theorem in equivariant Chow, and \cite[Theorem~1.4]{KatzPayne} for its extension to singular toric varieties.	
\end{proof}

Finally, recall that $\updelta \colon M \to A^1 (X)$ is obtained by taking the first Chern class of the mixing collection $L \colon M \to \Pic X$. This extends multiplicatively to a map:
\[ \updelta \colon \on{Sym} M \to A^\star X.\]
We now describe \eqref{eqn: map PP to MW}. Since this map is $A^\star X$-linear, it suffices to describe it on the subring $\on{PP}^\star(\Sigma) \otimes \mathbbm{1}_X \subseteq \on{PP}^\star(\Sigma) \otimes A^\star X$.

\begin{theorem} \label{prop: PP to MW} Fix a homogeneous piecewise polynomial $f \in \on{PP}^k(\Sigma)$. Then the Minkowski weight $W_f$ associated to $f \otimes \mathbbm{1}_X$ under \eqref{eqn: map PP to MW} is given by
\begin{equation} \label{eqn: PP to MW formula} W_f(\uptau) = \updelta(R_f(\uptau)) \cap [X]\end{equation}
for every $\uptau \in \Sigma$.
\end{theorem}
\begin{proof}[Proof of \Cref{prop: PP to MW}]
Given a cone $\uptau \in \Sigma$ consider the following diagram, in which the horizontal equalities are established in \Cref{prop: quotient bundle by dense torus}:
\[
\begin{tikzcd}
Y(\uptau) \ar[r] \ar[d,hook,"i"] \ar[rd,phantom,"\square"] & {[Y(\uptau)/T]} \ar[d,hook,"j"] \ar[r,equal] & X \times {[Z(\uptau)/T]} \ar[d,hook,"j"] \ar[r] \ar[rd,phantom,"\square"] & {[Z(\uptau)/T]} \ar[d,hook,"k"] \\
Y \ar[r,"u"] \ar[d,"p"] \ar[rrd,phantom,"\square"] & {[Y/T]} \ar[r,equal] & X \times {[Z/T]} \ar[d] \ar[r,"v"] \ar[rd,phantom,"\square"] & {[Z/T]} \ar[d,"q"] \\
X \ar[rr] \ar[rrr,bend right=20pt,"L"] & & X \times \Bcal T \ar[r] & \Bcal T.
\end{tikzcd}
\]
We have $A^\star[Z/T] = A^\star_T Z = \on{PP}^\star(\Sigma)$ and $A^\star [Y/T] = A^\star_T Y = \on{PP}^\star(\Sigma) \otimes A^\star X$. The subring $\on{PP}^\star (\Sigma) \otimes \mathbbm{1}_X \subseteq \on{PP}^\star(\Sigma) \otimes A^\star X$ is precisely the image of the pullback along $v$:
\[ v^\star A^\star[Z/T] \subseteq A^\star [Y/T].\]
We are given a piecewise polynomial $f \in \on{PP}^\star(\Sigma) = A^\star[Z/T]$ whose non-equivariant limit is $u^\star v^\star f \in A^\star Y$. The value of the corresponding Minkowski weight on a cone $\uptau \in \Sigma$ is given by Theorem~\ref{thm: Minkowski weights} as:
\[
W_f(\uptau) = p_\star \left( u^\star v^\star f \cap i_\star [Y(\uptau)] \right) = p_\star u^\star v^\star \left( f \cap k_\star [Z(\uptau)/T] \right) = L^\star q_\star \left( f \cap k_\star[Z(\uptau)/T] \right).
\]
The localisation theorem in equivariant Chow (see \cite[Theorem~2]{EdidinGrahamLocalisation} and \cite[Theorem~1.4]{KatzPayne}) applied to the action $T \curvearrowright Z(\uptau)$ gives the following formula
\[ q_\star \left(f \cap k_\star [Z(\uptau)/T] \right) = R_f(\uptau) \cap [\Bcal T] \]
from which we conclude:
\[
W_f(\uptau) = L^\star \left( R_f(\uptau) \cap [\Bcal T] \right) = \updelta (R_f(\uptau)) \cap [X].\qedhere
\]
\end{proof}

\begin{example} \label{ex: MW to PP} We continue Examples~\ref{ex: MW} and \ref{ex: MW product}. Recall the labelling of the cones of $\Sigma$:	
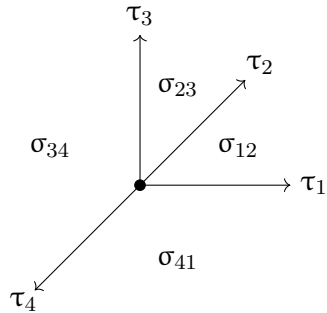
\[
\begin{tikzpicture}
	\draw(0,0) [fill=black] circle[radius=2pt];
	\draw[->] (0,0) -- (2,0);
	\draw (2,0) node[right]{$\uptau_1$};
	\draw[->] (0,0) -- (1.4,1.4);
	\draw (1.6,1.6) node{$\uptau_2$};
	\draw[->] (0,0) -- (0,2);
	\draw (0,2) node[above]{$\uptau_3$};
	\draw[->] (0,0) -- (-1.4,-1.4);
	\draw (-1.55,-1.55) node{$\uptau_4$};
	\draw (1.3,0.5) node{$\upsigma_{12}$};
	\draw (0.5,1.3) node{$\upsigma_{23}$};
	\draw (-1.2,0.5) node{$\upsigma_{34}$};
	\draw (0.5,-1) node{$\upsigma_{41}$};
\end{tikzpicture}
\]
Let $x_1,x_2$ denote the standard coordinates on $N=\Z^2$ and consider the following piecewise polynomial $f \in \operatorname{PP}^2(\Sigma)$:
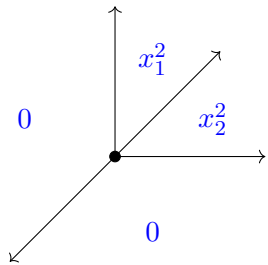
\[
\begin{tikzpicture}
	\draw(0,0) [fill=black] circle[radius=2pt];
	\draw[->] (0,0) -- (2,0);
	\draw[->] (0,0) -- (1.4,1.4);
	\draw[->] (0,0) -- (0,2);
	\draw[->] (0,0) -- (-1.4,-1.4);
	\draw (1.3,0.5) node[blue]{$x_2^2$};
	\draw (0.5,1.3) node[blue]{$x_1^2$};
	\draw (-1.2,0.5) node[blue]{$0$};
	\draw (0.5,-1) node[blue]{$0$};
\end{tikzpicture}
\]
We compute the associated Minkowski weight $W_f$. The relevant maximal cones are $\upsigma_{12}$ and $\upsigma_{23}$ and the relevant equivariant multiplicities are:
\begin{alignat*}{5}
& e_{\upsigma_{12}0} = \dfrac{1}{(x_1-x_2)x_2}, \qquad && e_{\upsigma_{12} \uptau_1} = \dfrac{1}{x_2}, \qquad && e_{\upsigma_{12}\uptau_2} = \dfrac{1}{x_1-x_2},\\[0.2cm]
& e_{\upsigma_{23}0} = \dfrac{1}{x_1(x_2-x_1)}, \qquad && e_{\upsigma_{23}\uptau_2} = \dfrac{1}{x_2-x_1}, \qquad && e_{\upsigma_{23}\uptau_3} = \dfrac{1}{x_1}.
\end{alignat*}
We use these to compute the residue sums. For the origin we have:
\[ R_f(0) = \dfrac{x_2^2}{(x_1-x_2)x_2} + \dfrac{x_1^2}{x_1(x_2-x_1)} = - \mathbbm{1}. \]
For the rays we have:
\begin{alignat*}{3} & R_f(\uptau_1) = \dfrac{x_2^2}{x_2} = x_2 \ && \Rightarrow \ W_f(\uptau_1) = \updelta(x_2) = \upalpha_2,\\[0.2cm]
& R_f(\uptau_2) = \dfrac{x_2^2}{x_1-x_2} + \dfrac{x_1^2}{x_2-x_1} = -x_1-x_2 \ && \Rightarrow \ W_f(\uptau_2) = \updelta(-x_1-x_2) = -\upalpha_1-\upalpha_2, \\[0.2cm]
& R_f(\uptau_3) = \dfrac{x_1^2}{x_1} = x_1 \ && \Rightarrow \ W_f(\uptau_3) = \updelta(x_1) = \upalpha_1.
\end{alignat*}
Finally for the maximal cones we have:
\begin{align*} R_f(\upsigma_{12}) = x_2^2 \ \Rightarrow \ W_f(\upsigma_{12}) = \updelta(x_2^2) = \upalpha_2^2, \\
R_f(\upsigma_{23}) = x_1^2 \ \Rightarrow \ W_f(\upsigma_{23}) = \updelta(x_1^2) = \upalpha_1^2.	
\end{align*}
Assembling the above classes, the Minkowski weight $W_f$ is given by:
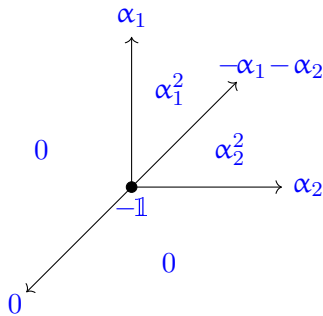
\[
\begin{tikzpicture}
	\draw(0,0) [fill=black] circle[radius=2pt];
	\draw(0,0) node[blue,below]{$-\!\mathbbm{1}$};
	\draw[->] (0,0) -- (2,0);
	\draw (2,0) node[right,blue]{$\upalpha_2$};
	\draw[->] (0,0) -- (1.4,1.4);
	\draw (1.85,1.6) node[blue]{$-\!\upalpha_1\!-\!\upalpha_2$};
	\draw[->] (0,0) -- (0,2);
	\draw (0,2) node[above,blue]{$\upalpha_1$};
	\draw[->] (0,0) -- (-1.4,-1.4);
	\draw (-1.55,-1.55) node[blue]{$0$};
	\draw (1.3,0.5) node[blue]{$\upalpha_2^2$};
	\draw (0.5,1.3) node[blue]{$\upalpha_1^2$};
	\draw (-1.2,0.5) node[blue]{$0$};
	\draw (0.5,-1) node[blue]{$0$};
\end{tikzpicture}
\]
The arguments in \Cref{ex: MW} can be used to show that this is the Minkowski weight Poincar\'e dual to $D_2^2$.
\end{example}

\footnotesize
\bibliographystyle{alpha}
\bibliography{Bibliography.bib}\medskip

\newcommand{\etalchar}[1]{$^{#1}$}
\begin{thebibliography}{KKMSD73}

\bibitem[ACM{\etalchar{+}}16]{AbramovichEtAlSkeletons}
D.~Abramovich, Q.~Chen, S.~Marcus, M.~Ulirsch, and J.~Wise.
\newblock Skeletons and fans of logarithmic structures.
\newblock In {\em Nonarchimedean and tropical geometry}, Simons Symp., pages
  287--336. Springer, 2016.

\bibitem[AIP{\etalchar{+}}12]{AltmannGeometry}
K.~Altmann, N.~O. Ilten, L.~Petersen, H.~S{\"u}{\ss}, and R.~Vollmert.
\newblock The geometry of {T}-varieties.
\newblock {\em Contributions to algebraic geometry, EMS Ser. Congr. Rep}, pages
  17--69, 2012.

\bibitem[Bor53]{Borel}
A.~Borel.
\newblock Sur la cohomologie des espaces fibr\'{e}s principaux et des espaces
  homog\`enes de groupes de {L}ie compacts.
\newblock {\em Ann. of Math. (2)}, 57:115--207, 1953.

\bibitem[Bot24]{BoteroGeneralized}
A.~Botero.
\newblock Generalized {M}inkowski weights and {C}how rings of {$T$}-varieties.
\newblock {\em Documenta Mathematica}, 29(4):831--861, 2024.

\bibitem[Bri96]{BrionPP}
M.~Brion.
\newblock Piecewise polynomial functions, convex polytopes and enumerative
  geometry.
\newblock In {\em Parameter spaces ({W}arsaw, 1994)}, volume~36 of {\em Banach
  Center Publ.}, pages 25--44. Polish Acad. Sci. Inst. Math., Warsaw, 1996.

\bibitem[Bri97]{BrionEquivariant}
M.~Brion.
\newblock Equivariant {C}how groups for torus actions.
\newblock {\em Transform. Groups}, 2(3):225--267, 1997.

\bibitem[Bro14]{BrownToricFibration}
J.~Brown.
\newblock Gromov-{W}itten invariants of toric fibrations.
\newblock {\em Int. Math. Res. Not. IMRN}, (19):5437--5482, 2014.

\bibitem[BV97]{BrionVergne}
M.~Brion and M.~Vergne.
\newblock An equivariant {R}iemann-{R}och theorem for complete, simplicial
  toric varieties.
\newblock {\em J. Reine Angew. Math.}, 482:67--92, 1997.

\bibitem[CGT24]{CoatesVirasoro}
T.~Coates, A.~Givental, and H.-H. Tseng.
\newblock Virasoro constraints for toric bundles.
\newblock {\em Forum Math. Pi}, 12:Paper No. e4, 28, 2024.

\bibitem[CLS11]{CLS}
D.~A. Cox, J.~B. Little, and H.~K. Schenck.
\newblock {\em Toric varieties}, volume 124 of {\em Graduate Studies in
  Mathematics}.
\newblock American Mathematical Society, Providence, RI, 2011.

\bibitem[CN22]{CarocciNabijou2}
F.~{Carocci} and N.~{Nabijou}.
\newblock {Tropical expansions and toric variety bundles}.
\newblock {\em arXiv e-prints}, July 2022.
\newblock arXiv:2207.12541.

\bibitem[CN24]{CarocciNabijou1}
F.~Carocci and N.~Nabijou.
\newblock Rubber tori in the boundary of expanded stable maps.
\newblock {\em J. Lond. Math. Soc. (2)}, 109(3):Paper No. e12874, 36, 2024.

\bibitem[DKU19]{DasguptaKhanUma}
J.~Dasgupta, B.~Khan, and V.~Uma.
\newblock Cohomology of torus manifold bundles.
\newblock {\em Math. Slovaca}, 69(3):685--698, 2019.

\bibitem[Dod24]{DodwellThesis}
E.~Dodwell.
\newblock Tropical geometry in torus bundles, 2024.
\newblock MPhil thesis, University of Cambridge.

\bibitem[EG98a]{EdidinGrahamIntersection}
D.~Edidin and W.~Graham.
\newblock Equivariant intersection theory.
\newblock {\em Invent. Math.}, 131(3):595--634, 1998.

\bibitem[EG98b]{EdidinGrahamLocalisation}
D.~Edidin and W.~Graham.
\newblock Localization in equivariant intersection theory and the {B}ott
  residue formula.
\newblock {\em Amer. J. Math.}, 120(3):619--636, 1998.

\bibitem[FMSS95]{FMSS}
W.~Fulton, R.~MacPherson, F.~Sottile, and B.~Sturmfels.
\newblock Intersection theory on spherical varieties.
\newblock {\em J. Algebraic Geom.}, 4(1):181--193, 1995.

\bibitem[FS97]{FultonSturmfels}
W.~Fulton and B.~Sturmfels.
\newblock Intersection theory on toric varieties.
\newblock {\em Topology}, 36(2):335--353, 1997.

\bibitem[Ful93]{FultonToric}
W.~Fulton.
\newblock {\em Introduction to toric varieties}, volume 131 of {\em Annals of
  Mathematics Studies}.
\newblock Princeton University Press, Princeton, NJ, 1993.
\newblock The William H. Roever Lectures in Geometry.

\bibitem[Ful98]{FultonBig}
W.~Fulton.
\newblock {\em Intersection theory}, volume~2 of {\em Ergebnisse der Mathematik
  und ihrer Grenzgebiete. 3. Folge. A Series of Modern Surveys in Mathematics
  [Results in Mathematics and Related Areas. 3rd Series. A Series of Modern
  Surveys in Mathematics]}.
\newblock Springer-Verlag, Berlin, second edition, 1998.

\bibitem[GSZ12]{GKM}
V.~Guillemin, S.~Sabatini, and C.~Zara.
\newblock Cohomology of {GKM} fiber bundles.
\newblock {\em J. Algebraic Combin.}, 35(1):19--59, 2012.

\bibitem[HKM24]{Hofscheier}
J.~Hofscheier, A.~Khovanskii, and L.~Monin.
\newblock Cohomology rings of toric bundles and the ring of conditions.
\newblock {\em Arnold Mathematical Journal}, 10(2):171--221, 2024.

\bibitem[JTY17]{JiangTsengYou}
Y.~Jiang, H.-H. Tseng, and F.~You.
\newblock The quantum orbifold cohomology of toric stack bundles.
\newblock {\em Lett. Math. Phys.}, 107(3):439--465, 2017.

\bibitem[{Ken}21]{KH-PT}
P.~{Kennedy-Hunt}.
\newblock {Logarithmic Pandharipande--Thomas Spaces and the Secondary
  Polytope}.
\newblock {\em arXiv e-prints}, December 2021.
\newblock arXiv:2112.00809. To appear \emph{Trans. Amer. Math. Soc.}

\bibitem[{Ken}23]{KHQuot}
P.~{Kennedy-Hunt}.
\newblock {The Logarithmic Quot space: foundations and tropicalisation}.
\newblock {\em arXiv e-prints}, August 2023.
\newblock arXiv:2308.14470.

\bibitem[KKMSD73]{KKMS}
G.~Kempf, F.~F. Knudsen, D.~Mumford, and B.~Saint-Donat.
\newblock {\em Toroidal embeddings. {I}}.
\newblock Lecture Notes in Mathematics, Vol. 339. Springer-Verlag, Berlin-New
  York, 1973.

\bibitem[{Kot}23]{KotoMirror}
Y.~{Koto}.
\newblock {A mirror theorem for non-split toric bundles}.
\newblock {\em arXiv e-prints}, October 2023.
\newblock arXiv:2310.09888.

\bibitem[KP08]{KatzPayne}
E.~Katz and S.~Payne.
\newblock Piecewise polynomials, {M}inkowski weights, and localization on toric
  varieties.
\newblock {\em Algebra Number Theory}, 2(2):135--155, 2008.

\bibitem[Kre99]{Kresch}
A.~Kresch.
\newblock Cycle groups for {A}rtin stacks.
\newblock {\em Invent. Math.}, 138(3):495--536, 1999.

\bibitem[MR20]{LogDT}
D.~{Maulik} and D.~{Ranganathan}.
\newblock {Logarithmic Donaldson-Thomas theory}.
\newblock June 2020.
\newblock arXiv:2006.06603.

\bibitem[MR23]{LogMNOP}
D.~{Maulik} and D.~{Ranganathan}.
\newblock {Logarithmic enumerative geometry for curves and sheaves}.
\newblock {\em arXiv e-prints}, November 2023.
\newblock arXiv:2311.14150.

\bibitem[Oda88]{Oda}
T.~Oda.
\newblock {\em Convex bodies and algebraic geometry: An introduction to the
  theory of toric varieties}, volume~15 of {\em Ergebnisse der Mathematik und
  ihrer Grenzgebiete (3) [Results in Mathematics and Related Areas (3)]}.
\newblock Springer-Verlag, Berlin, 1988.
\newblock Translated from the Japanese.

\bibitem[Oh21]{OhGIT}
J.~Oh.
\newblock Quasimaps to {GIT} fibre bundles and applications.
\newblock {\em Forum Math. Sigma}, 9:Paper No. e56, 39, 2021.

\bibitem[Pay06]{PaynePP}
S.~Payne.
\newblock Equivariant {C}how cohomology of toric varieties.
\newblock {\em Math. Res. Lett.}, 13(1):29--41, 2006.

\bibitem[Ran22]{RangExpansions}
D.~Ranganathan.
\newblock Logarithmic {G}romov-{W}itten theory with expansions.
\newblock {\em Algebr. Geom.}, 9(6):714--761, 2022.

\bibitem[Ros89]{Rossmann}
W.~Rossmann.
\newblock Equivariant multiplicities on complex varieties.
\newblock {\em Ast\'{e}risque}, 173-174:313--330, 1989.

\bibitem[SU03]{SankaranUma}
P.~Sankaran and V.~Uma.
\newblock Cohomology of toric bundles.
\newblock {\em Comment. Math. Helv.}, 78(3):540--554, 2003.

\bibitem[{Tsc}23]{TschanzExpansions}
C.~{Tschanz}.
\newblock {Expansions for Hilbert schemes of points on semistable
  degenerations}.
\newblock {\em arXiv e-prints}, October 2023.
\newblock arXiv:2310.08987.

\bibitem[Vak24]{Vakil}
R.~Vakil.
\newblock The {R}ising {S}ea: {F}oundations of {A}lgebraic {G}eometry.
\newblock \url{https://math.stanford.edu/~vakil/216blog/FOAGjul2724public.pdf},
  July 2024.

\end{thebibliography}

\noindent Francesca Carocci, University of Rome Tor Vergata, Italy. \href{mailto:carocci@mat.uniroma2.it}{carocci@mat.uniroma2.it} \\
\noindent Leonid Monin, EPFL, Switzerland. \href{mailto:leonid.monin@epfl.ch}{leonid.monin@epfl.ch} \\
\noindent Navid Nabijou, Queen Mary University of London, UK. \href{mailto:n.nabijou@qmul.ac.uk}{n.nabijou@qmul.ac.uk}

\end{document}